\documentclass[12pt,reqno]{amsart}

\usepackage{amssymb,amsmath,amsbsy,eucal,amsfonts,mathrsfs,latexsym}
\usepackage{graphicx,verbatim,psfrag,epsfig,enumerate}
\usepackage{stmaryrd}
\usepackage{lscape}
\usepackage{subfig}

\newtheorem{theorem}{Theorem}[section]
\newtheorem{lemma}[theorem]{Lemma}

\theoremstyle{definition}
\newtheorem{definition}[theorem]{Definition}

\newtheorem{remark}{{\it Remark}}[section]

\numberwithin{equation}{section}

\newcommand{\nn}{\nonumber}

\newcommand{\Be}{{\boldsymbol{e}}}

\newcommand{\Bn}{{\boldsymbol{n}}}

\newcommand{\Bp}{{\boldsymbol{p}}}
\newcommand{\Bq}{{\boldsymbol{q}}}
\newcommand{\Br}{{\boldsymbol{r}}}

\newcommand{\Bv}{{\boldsymbol{v}}}

\newcommand{\Bx}{{\boldsymbol{x}}}

\newcommand{\BH}{{\boldsymbol{H}}}

\newcommand{\BL}{{\boldsymbol{L}}}

\newcommand{\BV}{{\boldsymbol{V}}}

\newcommand{\Ce}{{\mathcal E}}

\newcommand{\Ci}{{\mathcal I}}

\newcommand{\Cn}{{\mathcal N}}

\newcommand{\Cp}{{\mathcal P}}

\newcommand{\Ct}{{\mathcal T}}

\newcommand{\bbeta}{{\boldsymbol{\beta}}}

\newtheorem{exm}{Example}[section]

\begin{document}

\title{Robust a Posteriori Error Estimates for HDG method for Convection-Diffusion Equations}


\author{Huangxin Chen}
\address{School of Mathematical Sciences and Fujian Provincial Key Laboratory on Mathematical Modeling \& High Performance Scientific Computing, Xiamen University, Fujian, 361005, P.R. China} 
\email{chx@xmu.edu.cn} 

\author{Jingzhi Li}
\address{Faculty of Science, South University of Science and Technology of China,Shenzhen, 518055, China}
\email{li.jz@sustc.edu.cn} 
\author{Weifeng Qiu}
\address{Department of Mathematics, City University of Hong Kong, 83 Tat Chee Avenue, Kowloon, Hong Kong, China} 
\email{weifeqiu@cityu.edu.hk}

\thanks{The authors would also like to thank the associate editor and all the referees for constructive criticism leading to 
a better presentation of the material in this paper. The first author would like to thank the support from the City University of 
Hong Kong where this work was carried out  during his visit, and he also thanks the supports from the NSF of China 
(Grant No. 11201394) and the NSF of Fujian Province (Grant No. 2013J05016). The work of the second author was supported 
by the NSF of China (Grant No. 11201453). The work of the third author was supported by the GRF of Hong Kong (Grant No. 9041980).}


\begin{abstract}
We propose a robust  a posteriori error estimator for the hybridizable discontinuous Galerkin (HDG) method for convection-diffusion equations with dominant convection. The reliability and efficiency of the estimator are established for the error measured in an energy norm. The energy norm is uniformly bounded even when the diffusion coefficient tends to zero. The estimators are robust in the sense that the upper and lower bounds of error are uniformly bounded 
with respect to the diffusion coefficient. A weighted test function technique and the Oswald interpolation are key ingredients in the analysis. Numerical results verify the robustness of the proposed a posteriori error estimator. 
\end{abstract}



\maketitle

\markboth{H. Chen, J. Li, W. Qiu}{Robust a posteriori error estimator for convection diffusion equations}




\section{Introduction}\label{introduction}
Given a bounded, polyhedral domain $\Omega \subset R^d (d=2,3)$, we consider the convection-diffusion equations
\begin{subequations}
\label{cd_eqs}
\begin{align}
\label{cd_eqs_1}
-\epsilon\Delta u + \boldsymbol{\beta}\cdot \nabla u + cu = &\; f  \quad \text{ in $\Omega$, } 
\\
\label{bd_eqs}
u = &\; g \quad \text{ on $\partial \Omega$.}
\end{align}
\end{subequations}
The data and the right-hand sides in (\ref{cd_eqs}) satisfy the following assumptions:
\begin{enumerate}
\item[(A1)] $0<\epsilon \leq 1$.
\item[(A2)] $ \boldsymbol{\beta} \in W^{1,\infty}(\Omega)^d$, $c \in L^{\infty}(\Omega)$, $f\in L^2(\Omega)$ 
and $g\in H^{\frac{1}{2}}(\partial \Omega)$.
\item[(A3)] $c-\frac{1}{2} {\rm div}\boldsymbol{\beta} \geq 0$.
\item[(A4)] There is a function $\psi \in W^{1,\infty}(\Omega)$ and a positive constant $b_0$ 
such that $\boldsymbol{\beta} \cdot \nabla \psi \geq b_0$. 
\end{enumerate}
Assumption (A1) includes the case of  the convection-dominated regime. According to \cite{AyusoMarini:cdf}, Assumption (A4) is satisfied if $\boldsymbol{\beta}$ has no closed curves and $\vert \boldsymbol{\beta}(x)\vert\neq 0 \text{ for all }x\in\Omega$.

It is well known that solutions of  (\ref{cd_eqs}) may develop layers (cf. \cite{Eckhaus72,FLRT83}). In particular, 
the solutions may have singular interior layer of width $O(\sqrt{\epsilon})$ or outflow layer of width $O(\epsilon)$. 
Standard numerical methods, e.g., standard finite element method or central finite difference method, are not robust 
when the quantity $\epsilon/\|\boldsymbol{\beta}  \|_{L^\infty(\Omega)}$ is small compared to the mesh size. 
In order to stabilize the numerical method, several remedies are proposed for addressing the issue, for instance, 
streamline diffusion method \cite{BrooksHughes}, 
residual free bubble methods  \cite{BrezziHughesMariniRussoSuli1999,BrezziMariniSuli2000,BurmanErn09}, 
local projection schemes \cite{KT11}, subgrid scale method \cite{Codina02,BC2006}, 
continuous interior penalty (CIP) methods \cite{Burman04,Burman07}, 
discontinuous Galerkin methods \cite{AyusoMarini:cdf,HoustonSchwabSuli2002,HughesScovazziBochevBuffa2006}, 
and recently discontinuous Petrov-Galerkin (DPG) methods
\cite{BroersenStevenson:mild_weak_cd,ChanHeuerTanDemkowicz:DPG_CD,DemkoHeuer:2013:DPG_cd}, 
HDG method \cite{FQZ13} and the first order least squares method \cite{CFLQ14}. 
One can refer to \cite{Roos08,Stynes:acta_cd} for more other stabilization techniques. 
But in order to capture the potential interior or outflow layer of the solutions to the problem (\ref{cd_eqs}), 
the local P\'eclet number $P_e = h \|\boldsymbol{\beta}  \|_{L^\infty(\Omega)} / \epsilon $ near the layers 
should be small enough, where $h$ is the mesh size. Hence it would be quite expensive for the stabilized 
numerical methods used on the quasi-uniform mesh to capture the layers when $\epsilon$ is small. If the mesh 
in the vicinity of the layers can be locally refined, the cost of numerical computations could be reduced. 
Therefore the adaptive finite element method is a natural choice for the efficient solution of  
convection-diffusion equations with dominant convection.

The adaptive finite element method based on a posteriori error estimates have been well established for second-order 
elliptic problems (cf. \cite{Ainsworth,Verfurth96}). In recent years the a posteriori error estimates are also extended to 
convection-diffusion equation. An early attempt was proposed by Eriksson and Johnson in \cite{EJ93}, using regularization 
and duality techniques. Verf\"{u}rth \cite{Verfurth98} proposed semi-robust estimators in the energy norm for the standard 
Galerkin approximation and the streamline upwind Petrov-Galerkin (SUPG) discretization. In \cite{Verfurth05} Verf\"{u}rth 
improved his results by giving the estimates which are robust with dominant convection in a norm incorporating the 
standard energy norm and a dual norm of the convective derivative. Very recently, Tobiska and Verf\"{u}rth \cite{TV2014} derived 
the same robust a posteriori error estimators for a wide range of stabilized finite element methods such as streamline 
diffusion methods, local projection schemes, subgrid scale technique and CIP method. However, the energy norm of error 
used in Verf\"{u}rth's estimates is defined through a dual norm which is not easy to compute. Sangalli \cite{Sangalli08} 
proposed different norms for the a posteriori error estimates that allow for robust estimators, but the analysis is only 
valid in the one dimensional case. As to other approaches for the robust error estimations, one can refer to \cite{Voh07} 
for mixed finite element methods, \cite{Voh08} for cell-centered finite volume scheme,  \cite{Alaoui07} for 
nonconforming finite element method,  \cite{Ern08,Ern10,SZ09,SZ11} for interior penalty discontinuous Galerkin method.

Discontinuous Galerkin (DG) methods have several attractive features compared with conforming finite element methods. 
For example, DG methods have elementwise conservation of mass, and they work well on arbitrary meshes. However, 
the dimension of the approximation DG space is much larger than the dimension of the corresponding conforming space. 
The HDG method \cite{Cockburn09,Cockburn10,KSC2012} was recently introduced to address this issue. HDG methods 
retain the advantages of standard DG methods and result in significant reduced degrees of freedom. New variables on 
all interfaces of mesh are introduced such that the numerical solution inside each element can be computed in terms of them, 
and the resulting algebraic system is only due to the unknowns on the skeleton of the mesh. In \cite{FQZ13}, the HDG method 
was proposed and analyzed for the problem (\ref{cd_eqs}) on shape-regular mesh. The stabilization parameter of the HDG method 
in \cite{FQZ13} can be determined clearly, meanwhile  the penalty parameters of the DG schemes \cite{AyusoMarini:cdf} need to be 
chosen empirically. Moreover, the condition number of stiffness matrix of a new way of implementing HDG method 
in \cite{FQZ13} was proven to be bounded by $O(h^{-2})$ and independent of the diffusion coefficient $\epsilon$. 
These properties are important for the efficient solution of the problem (\ref{cd_eqs}) and encourage us to consider 
the corresponding HDG method on adaptive meshes.

The a posteriori error analysis for the HDG method for second order elliptic problems has been presented in \cite{Cockburn12,Cockburn13}, where the error incorporates only the flux and a postprocessed solution used in the estimators. To our best knowledge, no a posteriori error estimates for the HDG discretizations of convection-diffusion problems have been studied in the literature so far. In this work, our objective is to show that the HDG scheme proposed in \cite{FQZ13} gives rise to robust a posteriori error estimates for the problem (\ref{cd_eqs}). In comparison with the postprocessing technique 
utilized in \cite{Voh07,Cockburn12,Cockburn13}, we establish the estimators without any postprocessed solution since the solution of 
 (\ref{cd_eqs}) is always not smooth and there is no superconvergence result for the HDG method when $\epsilon \ll 1$.

We notice that the a posteriori error estimators for nonconforming finite element method \cite{Alaoui07} and 
interior penalty DG method \cite{Ern08,Ern10} are only semi-robust in the sense that they yield lower and upper bounds of 
the error which differ by a factor equal at most to $\mathcal{O}(\epsilon^{-1/2})$. In \cite{SZ09,SZ11}, the a posteriori error estimator 
is robust in the sense that the ratio of the constants in the upper and lower bounds of error is independent of the diffusion coefficient.
However, the energy norm of error in \cite{SZ09,SZ11} contains the jump term 
$\left( h_{F}\epsilon^{-1}\| \llbracket  u_h \rrbracket  \|^2_{0,F} \right)^{1/2}$ on each interior interface of meshes.  
In contrast, the a posteriori error estimator in this paper is robust based on the energy norm in (\ref{energy-error}), 
which contains a jump term $\left( \gamma_F\| \llbracket  u_h \rrbracket  \|^2_{0,F} \right)^{1/2}$ instead. One can refer to (\ref{A_gamma}) for the definition of paparemeter $\gamma_F$. So, our a posteriori error estimator will not enlarge the error estimate too much as the error estimator in \cite{SZ09,SZ11},  
when the mesh size is not small compared with the diffusion coefficient.

To derive the reliability and efficiency of the estimators for the error measured in an energy norm which incorporates a scaling flux 
and the scalar solution of the HDG discretization, two techniques are utilized. The first one is to use the Oswald interpolation operator 
to approximate a discontinuous polynomial by a continuous and piecewise polynomial function and to control the approximation by 
the jumps (cf. \cite{OP03,OP07}). For most of a posteriori error estimates mentioned above (e.g. \cite{Verfurth05,Voh07,TV2014,Ern08,Ern10,SZ09,SZ11}), the analysis only gives the estimates for the energy error 
without  $L^2$-error of the scalar solution when $c-\frac{1}{2} {\rm div}\boldsymbol{\beta} = 0$. The second one is to address 
this issue and to employ a weighted function to derive the estimates for the error which contains $L^2$-error of the scalar solution. 
This idea goes back to N\"{a}vert's work \cite{Navert1982} for convection-diffusion problems and was used to obtain the $L^2$-stability of the original DG method for pure hyperbolic equation \cite{RH1973} and 
extended to convection-diffusion equations using the IP-DG method \cite{Guzman2006,AyusoMarini:cdf}, the HDG method \cite{FQZ13} 
and the first order least squares method \cite{CFLQ14}. 

In the numerical experiments, the convection-diffusion problems with interior or outflow layers are tested based on the 
proposed a posteriori error estimator. The robustness of the a posteriori error estimator based on the HDG method 
is observed for the problems with different diffusion coefficient. We also find that the convergence of the adaptive HDG method is 
almost optimal, i.e.,  the convergence rate is almost $O(N^{-s})$, where $N$ is the number of elements, $s$ depends on 
the polynomial order $p$.


The outline of the paper is as follows: We introduce some notations, the HDG method, a posteriori error estimator and main results in the next section.
In section 3, we collect some auxiliary results for analysis. Section 4 and section 5 are devoted to the proofs of reliability and efficiency, respectively.
In the final section, we give some numerical results to confirm our theoretical analysis.

\section{Notation, HDG method, error estimator, and main results}\label{HDG_sec2}
In this section, we begin with some basic notation and hypotheses of meshes. Secondly, we introduce  
the HDG method for (\ref{cd_eqs}) in \cite{FQZ13}. Then, we define the corresponding 
a posteriori error estimator. Finally, we give the main results of reliability and efficiency.  

\subsection{Notation and the mesh}
Let $\mathcal{T}_{h}$ be a conforming, shape-regular simplicial triangulation of $\Omega$.
For any element $T\in \mathcal{T}_{h}$, $\partial{T}$ denotes the set of its edges in the two
dimensional case and of its faces in the three dimensional case.
Elements of $\partial{T}$ will be generally referred to as faces,
{regardless of dimension}, and denoted by $F$.
We define $\partial\mathcal{T}_{h} := \{\partial T: T\in\mathcal{T}_{h}\}$.
We denote by $\mathcal{E}_{h}$ the set of all faces in the triangulation (the skeleton), while 
the set of all interior {(boundary)}
faces of the triangulation will be denoted $\mathcal{E}_{h}^{0}$ {($\mathcal{E}_{h}^{\partial}$)}. 
Correspondingly, we refer to $\Cn_h$ the set of vertices and to $\Cn_h^0$ the set of interior vertices. 
For any  $T\in \mathcal T_h $, let $h_T$ be the diameter of element $T$. Similarly, for any $F\in \Ce_h$, we define $h_{F} := {\rm diam}(F)$. 
Throughout this paper, we use the standard notations and definitions for Sobolev spaces (see, e.g., Adams\cite{Adams}). We also use the notation 
$\| \cdot \|_{0,D}$ and $\| \cdot  \|_{0,\Gamma}$ to denote the $L^2$-norm on the elements $D$ and faces $\Gamma$, respectively.

\subsection{The HDG method}
The HDG method is based on a first order formulation of the convection-diffusion equation (\ref{cd_eqs}),  
which can be rewritten in a mixed form as finding $(\Bq, u)$ such that
\begin{subequations}
\label{HDG1}
\begin{align}
\epsilon^{-1}{\Bq}+\nabla u&=0 \qquad {\rm in }\ \Omega, \\
{\rm div}\, {\Bq} + \bbeta \cdot  \nabla u + cu&=f  \qquad {\rm in }\ \Omega, \label{mixed-eqs}\\
u&=g \qquad {\rm on }\
\partial\Omega. \label{HDG3}
\end{align}
\end{subequations}

For any element $T\in \mathcal{T}_{h}$ and any face $F\in\mathcal{E}_{h}$, we define
$${\BV}(T):=(\mathcal P_p(T))^d,\qquad W(T):=\mathcal P_p(T),\qquad M(F):=\mathcal P_p(F),$$
where $\mathcal P_p(S)$ is the space of polynomials of total
degree not larger than $p\geq 1$ on $S$. The finite element spaces are given by
\begin{align*}
{\BV}_h:&=\{{\Bv}\in {\BL}^2(\Omega)\,  : \, {\Bv}|_T\in {\BV}(T)\ {\rm for\  all }\ T\in \mathcal T_h\},\\
W_h:&=\{w\in L^2(\Omega)\, : \, w|_T\in W(T)\ {\rm for\ all }\ T\in \mathcal T_h\},\\
M_h:&=\{ \mu\in L^2(\mathcal E_h)\, : \,  \mu|_F\in M(F)\ {\rm for\ all }\ F\in \mathcal E_h\},\\
M_h(g):&=\{ \mu \in M_h \, : \, \int_{\partial\Omega} (\mu-g)\xi ds =0 {\rm \ for\ all }\ \xi \in M_h \},
\end{align*}
where ${\BL}^2(\Omega):=( L^2(\Omega))^d$ and $L^2(\mathcal
E_h):=\Pi_{F\in \mathcal E_h}L^2(F)$. 

The HDG method seeks finite element
approximations $({\Bq}_h, u_h,\widehat u_h)\in {\BV}_h\times
W_h\times M_h$ satisfying
\begin{subequations}
\label{PD}
\begin{align}
\label{P1}
& (\epsilon^{-1} {\Bq}_h, \Br)_{\mathcal T_h}-(u_h, {{\rm div \, {\Br}}})_{\mathcal T_h}+\langle \widehat u_h, {{\Br}\cdot {\Bn}}\rangle_{\partial \mathcal T_h}=0,\\
\label{P2}
& -({\Bq}_h +\bbeta u_h ,{\nabla w})_{\mathcal T_h}+ ((c-{\rm div}\, \bbeta) u_h, w)_{\mathcal T_h}+\langle (\widehat {\Bq}_h + \widehat{\bbeta u_h})\cdot {\Bn} ,w\rangle_{\partial \mathcal T_h}\\
\nonumber
&\qquad \qquad \qquad =(f, w)_{\mathcal T_h} ,\\
\label{P3}
& \langle  \widehat u_h,\mu\rangle_{\partial \Omega}=\langle g,  \mu\rangle_{\partial \Omega},\\
\label{P4} 
& \langle (\widehat {\Bq}_h + \widehat{\bbeta u_h})\cdot {\Bn},
\mu\rangle_{\partial \mathcal T_h \backslash
 \partial \Omega}=0,
\end{align}
\end{subequations}
for all $({\Br},w,\mu)\in {\BV}_h \times W_h \times M_h$, where the normal component of numerical flux
$(\widehat {\Bq}_h + \widehat{\bbeta u_h})\cdot \Bn$ is given by
\begin{align}\label{numerical-flux}
(\widehat {\Bq}_h + \widehat{\bbeta u_h})\cdot \Bn ={\Bq}_h\cdot \Bn + (\bbeta\cdot \Bn) \widehat{u}_h +\tau(u_h-\widehat u_h) \qquad {\rm on } \
\partial \mathcal T_h,
\end{align}
and the stabilization function $\tau$ is a piecewise, nonnegative constant defined on $\partial \mathcal T_h$.
Here, we define $\left(\eta,\zeta\right)_{\mathcal{T}_{h}} := \sum_{K \in \mathcal{T}_{h}} \int_K \eta\,\zeta\,\mathrm{dx},$ 
and $\langle \eta, \zeta \rangle_{\partial\mathcal{T}_{h}} := \sum_{K \in \mathcal{T}_{h}} \int_{\partial K} \eta\,\zeta\,\mathrm{ds}$.
One of the advantages of the HDG method is the elimination of both $\Bq_h$ and $u_h$ from the system (\ref{PD}) to obtain a formulation in terms of numerical trace $\widehat{u}_h$ only, one can refer to \cite{Cockburn09JSC,FQZ13,KSC2012,Nguyen09} for the implementation. 

The stabilization function $\tau$ in (\ref{numerical-flux}) is chosen as 
\begin{equation}
\label{tau_hdg}
\tau |_{F} = \max ( \sup_{\Bx \in F} \bbeta(\Bx) \cdot \Bn, 0  ) + \min( \rho_0 \frac{\epsilon}{h_T}, 1 ) \qquad \forall F\in \partial T,  \ T \in \Ct_h.
\end{equation}
Here, $0<\rho_0 \leq 1$. We emphasize that the choice of $\tau$ 
in (\ref{tau_hdg}) is the second type of stabilization function in \cite{FQZ13}. 
According to \cite{ChenCockburnHDGI}, the HDG method (\ref{PD}) with stabilization function (\ref{tau_hdg}) has a unique solution.
Compared with a recent work \cite{CFLQ14}, the intrinsic idea of choosing $\tau$ in the above HDG method is similar to the strategy of setting the ultra-weakly imposed boundary condition in the first order least squares method for (\ref{cd_eqs}).

\subsection{A posteriori error estimator}

We define  the elementwise residual function $R_{h}$ as
\begin{equation}
\label{elem_residual}
R_{h}|_{T} = f - \text{div} \boldsymbol{q}_{h} - \boldsymbol{\beta}\cdot \nabla u_{h} - c u_{h},\quad \forall T\in\mathcal{T}_{h}.
\end{equation}
We define 
\begin{align}
\label{A_alpha}
\alpha_S = \min\{h_S{\epsilon^{-\frac{1}{2}}} ,1\},
\end{align}
where $S$ can be any element $T\in\mathcal{T}_{h}$ or any face $F\in\mathcal{E}_{h}$.
In addition, for any $F\in\mathcal{E}_{h}$, we introduce 
\begin{align}
\label{A_gamma}
\gamma_F =\min\{ \frac{\epsilon}{h_F}+(\frac{h_F}{\epsilon}+\epsilon^{-\frac{1}{2}}\alpha_F) \|\boldsymbol{\beta}\|_{L^\infty(F)}  +h_F ,\frac{ \epsilon + \|\boldsymbol{\beta}\|_{L^\infty(F)}  }{h_F}  + h_F\}.
\end{align}

Now, we are ready to introduce the a posteriori error estimator in the following.
\begin{definition}
\label{def_estimator}(A posteriori error estimator)
The a posteriori error estimator is defined as
\begin{subequations}
\label{estimators}
\begin{align}
\label{estimator_eta}
\eta & = \left( \Sigma_{T\in\mathcal{T}_{h}}\eta_{T}^{2} + \Sigma_{F\in \mathcal{E}_{h}^{0}}(\eta_{F}^{0})^{2}
+\Sigma_{F\in \mathcal{E}_h^{\partial}}(\eta_{F}^{\partial})^{2}\right)^{\frac{1}{2}}\text{ where }\\
\label{estimator1}
\eta_{T} & = \left( \alpha_{T}^{2}\Vert R_{h}\Vert_{0,T}^{2} + \epsilon^{-1} \Vert \boldsymbol{q}_{h}+\epsilon\nabla u_{h}\Vert_{0,T}^{2}
\right)^{\frac{1}{2}}, \quad \forall T\in \mathcal{T}_{h},\\
\label{estimator2}
\eta_{F}^{0} & = \left( \epsilon^{-\frac{1}{2}} \alpha_F \| \llbracket \Bq_h \cdot\Bn \rrbracket  \|^2_{0,F} + \gamma_F\| \llbracket  u_h \rrbracket  \|^2_{0,F}
 \right)^{\frac{1}{2}}, \quad \forall F\in  \mathcal{E}_h^{0},\\
\label{estimator3}
\eta_{F}^{\partial} & =\gamma_F^{\frac{1}{2}}\|  u_h - g  \|_{0,F}, \quad \forall F\in \mathcal{E}_h^{\partial}.
\end{align}
\end{subequations}
Here, for any interior face $F=\partial T^{+} \cap \partial T^{-}$ in $\mathcal{E}_{h}^{0}$, we define the jump of scalar function $\phi$  
and the jump of normal component of vector field $\boldsymbol{\sigma}$ by
\begin{align*}
\llbracket\phi\rrbracket  = \phi^{+} - \phi^{-},\quad  
\llbracket \boldsymbol{\sigma}\cdot \boldsymbol{n}\rrbracket = \boldsymbol{\sigma}^{+}\cdot \boldsymbol{n}^{+} 
  + \boldsymbol{\sigma}^{-}\cdot \boldsymbol{n}^{-}, \text{ respectively}.
\end{align*}
\end{definition}

\subsection{Reliability and efficiency of a posteriori error estimator}
From now on, we use $C_{0}, C_{1}, C_{2}$ to denote generic constants, which are independent of the diffusion coefficient $\epsilon$ 
and the mesh size.

For any $(\Bp,w) \in \BH^1(\Ct_h) \times H^1(\Ct_h)$, we define the energy norm 
\begin{align}
\label{energy-error}
& \interleave (\Bp,w)\interleave^2_h \\
\nn
= &\sum_{T\in \Ct_h} \left(  \epsilon^{-1} \| \Bp \|^2_{0,T} + \|w\|^2_{0,T} + \epsilon\|\nabla w\|^2_{0,T} + \alpha^2_T \| {\rm div}\Bp +\bbeta \cdot \nabla w \|^2_{0,T}  \right) \\
\nn
&\quad + \sum_{F\in \Ce^0_h}  \left(  \epsilon^{-\frac{1}{2}} \alpha_F\| \llbracket \Bp\cdot \Bn\rrbracket \|^2_{0,F} + \gamma_F \| \llbracket w \rrbracket \|^2_{0,F} \right) + \sum_{F\in \Ce^{\partial}_h} \gamma_F \| w  \|^2_{0,F}.
\end{align}

We outline  main results by showing the reliability and efficiency of the a posteriori error estimator, respectively, 
in the following theorems. Using the techniques developed in this paper, we would like to emphasize that all the a posteriori error estimates in this paper will also hold for the mixed hybrid method in \cite{Egger2010}.
\begin{theorem} (Reliability)
\label{thm_reliability}
If the Dirichlet data $g$ is contained in $C(\partial \Omega)\cap M_{h}|_{\partial \Omega}$, then 
\begin{align}
\label{ineq_reliability}
\interleave (\boldsymbol{q}-\boldsymbol{q}_{h},u-u_{h})\interleave_h \leq C_{0}\eta.
\end{align}
\end{theorem}

\begin{remark}
In the diffusion dominated case $\epsilon = O(1), \boldsymbol{\beta} = 0$ and $c\geq 0$, the parameter $\gamma_F$ in (\ref{A_gamma}) would be $\frac{1}{h_F}$, and the proposed a posterior error estimators coincide with the ones in \cite{Cockburn13}. In this particular case, the total energy norm can be defined as $ \interleave (\Bp,w) \interleave^2 = \sum_{T\in \Ct_h} \left( \| \Bp \|^2_{0,T} + \|w\|^2_{0,T} + \|\nabla w\|^2_{0,T} \right) $.

\end{remark}

\begin{theorem} (Efficiency)
\label{thm_efficiency}
We have 
\begin{subequations}
\label{ineq_efficiency}
\begin{align}
\label{in_face_efficiency}
\eta_{F}^{0} & \leq C_{1} \left( \epsilon^{-\frac{1}{2}} \alpha_F \| \llbracket (\Bq - \boldsymbol{q}_{h}) \cdot\Bn \rrbracket  \|^2_{0,F} 
+ \gamma_F\| \llbracket  u - u_{h} \rrbracket  \|^2_{0,F} \right)^{\frac{1}{2}},\quad \forall F \in \mathcal{E}_{h}^{0},\\
\label{boundary_face_efficiency}
\eta_{F}^{\partial} & \leq C_{1} \left(\gamma_F\|  g-u_h   \|^2_{0,F} \right)^{\frac{1}{2}},\quad \forall F \in \mathcal{E}_{h}^{\partial},\\
\label{elem_efficiency}
\eta_{T} & \leq C_{1} (\epsilon^{-1} \| \boldsymbol{q} -\boldsymbol{q}_{h} \|^2_{0,T} 
+ \| u - u_{h}\|^2_{0,T} + \epsilon\|\nabla (u - u_{h})\|^2_{0,T} \\
\nonumber
&\qquad  + \alpha^2_T \| {\rm div}(\boldsymbol{q} - \boldsymbol{q}_{h}) +\bbeta \cdot \nabla (u - u_{h}) \|^2_{0,T}
+{osc}^2_h(R_h,T))^{\frac{1}{2}}, 
\quad \forall T \in \mathcal{T}_{h}.
\end{align}
\end{subequations}
Here, the data oscillation term ${osc}^2_h(R_h,T) : = \alpha^2_T \| R_h -P_W R_h \|^2_{0,T}$, and 
$P_{W}$ is the $L^{2}$ orthogonal projection onto $W_{h}$.
\end{theorem}

In addition, for any $F\in \mathcal{E}_{h}^{0}$ with $h_{F}\leq \mathcal{O}(\epsilon)$ or any $T\in \mathcal{T}_{h}$ with $h_{T}\leq \mathcal{O}(\epsilon)$, we have  
the following efficiency results with explicit dependence on $\epsilon$.
\begin{theorem} (Efficiency on refined element)
\label{thm_refined_efficiency}
For any $F\in \mathcal{E}_{h}^{0}$, if $h_{F}\leq \mathcal{O}(\epsilon)$, then
\begin{align}
\label{ineq_refined_efficiency}
(\eta_{F}^{0})^{2} & \leq C_{2} \sum_{T \in \omega_F} (  \epsilon^{-1}\| \Bq-\Bq_h \|^2_{0,T} + \epsilon\| \nabla(u-u_h) \|^2_{0,T}\\
\nonumber
& \qquad \qquad  +\|u-u_h\|^2_{0,T} + osc^2_h(R_h,T)  ).
\end{align}
For any $T\in \mathcal{T}_{h}$, if $h_{T}\leq \mathcal{O}(\epsilon)$, then
\begin{align}
\label{ineq_refined_efficiency2}
\eta_{T}^{2} & \leq C_{2} \big(  \epsilon^{-1}\| \Bq-\Bq_h \|^2_{0,T} + \epsilon\| \nabla(u-u_h) \|^2_{0,T}\\
\nonumber
& \qquad \qquad  +\|u-u_h\|^2_{0,T} + osc^2_h(R_h,T)  \big).
\end{align}
Here, $\omega_F$ is  the union of elements sharing the common face $F$, and ${osc}^2_h(R_h,T)$ is 
the data oscillation term introduced in Theorem~\ref{thm_efficiency}.
\end{theorem}

\begin{remark}
For the diffusion dominated case $\epsilon = O(1)$, $\boldsymbol{\beta} = 0$ and $c\geq 0$, the above efficiency estimates (\ref{ineq_refined_efficiency}) and (\ref{ineq_refined_efficiency2}) also coincide with the efficiency results in \cite{Cockburn13}.

\end{remark}

\section{Auxiliary results}
We collect some auxiliary results in this section for the proof of reliability and efficiency. 

For every element $T\in \Ct_h$, we denote $\Omega_T$ by the union of all elements that share at least one point with $T$. For any face $F\in \Ce_h$, 
the set $\Omega_F$ is defined analogously, meanwhile, $\omega_F$ is defined to be the union of elements sharing the common face $F$.
The following Cl\'{e}ment-type interpolation is crucial for the proof of reliability.
\begin{definition}\label{define-clement}
(cf. \cite{Verfurth96}) One can define a linear mapping $\pi_h: \ L^1(\Omega) \rightarrow W^c_{1,h} \cap H^1_0(\Omega) $ via
\begin{align}
\pi_h v := \sum_{z \in \Cn^0_h} \left(   \frac{1}{|\Omega_z|}  \int_{\Omega_z} v \, dx\right) \phi_z, \nn
\end{align}
where $\phi_z$ is $P_1$ nodal bases function for every vertex $z \in \Cn_h^0$, $\Omega_z$ is the support of a nodal bases function $\phi_z$ which consists of all elements that share the vertex $z$, and $W^c_{1,h}$ is the corresponding conforming $P_1$ finite element space defined by
$
W^c_{1,h} := \{ w \in C(\Omega) \, : \, w |_{T} \in \Cp_1(T), T \in \Ct_h \} .
$
\end{definition}

The interpolation $\pi_{h}$ in Definition~\ref{define-clement} has the following approximation properties.
\begin{lemma}\label{ax_re_1}
{\rm (cf. \cite{Verfurth98,Verfurth05})} For any $T \in \Ct_h$ and $F\in \Ce^0_h$, the following estimates hold for any function $v \in H^1_0(\Omega)$:
\begin{subequations}
\begin{align}
\| (I-\pi_h)v \|_{0,T} &\leq C_1 \alpha_T \left( \Vert v\Vert_{0, \Omega_{T}}^{2} + \epsilon\Vert \nabla v\Vert_{0,\Omega_{T}}^{2}\right)^{\frac{1}{2}},\\
\| (I-\pi_h)v \|_{0,F} &\leq C_2  \epsilon^{-\frac{1}{4}} \alpha_F^{\frac{1}{2}}   \left( \Vert v\Vert_{0, \Omega_{T}}^{2} + \epsilon\Vert \nabla v\Vert_{0,\Omega_{T}}^{2}\right)^{\frac{1}{2}}.
\end{align}
\end{subequations}
\end{lemma}

\begin{remark}
The proof of Lemma \ref{ax_re_1} can be obtained by the Lemmas 3.1 and 3.2 in \cite{Verfurth98}.
For the particular case $c-\frac{1}{2} {\rm div}\, \bbeta=0$, the constant $\alpha_S $ is set to be $h_S{\epsilon^{-\frac{1}{2}}} $ for any $S = T\in \Ct_h$ or $F\in \Ce^0_h$ in \cite{Verfurth05}, and the associated energy error excludes the $L^2$-error $\|u-u_h\|_{0,\Ct_h}$. In this paper, a weighted function technique 
used in \cite{AyusoMarini:cdf} shall be employed, such that we can also obtain the estimates for the energy error including the $L^2$-error $\|u-u_h\|_{0,\Ct_h}$ even when $c-\frac{1}{2} {\rm div}\, \bbeta=0$. Hence, $\alpha_S$ is always set as (\ref{A_alpha}).
\end{remark}


For the approximation of function in $W_{h}$ and the Dirichlet boundary data $g$ by continuous finite element space, we 
need to introduce Oswald interpolation.
If $g$ is contained in $C(\partial \Omega)\cap M_{h}|_{\partial \Omega}$, then the continuous finite element space 
$W_{h,g}^c  := \{ w \in C(\Omega) \, : \, w |_{T} \in \Cp_p(T), T \in \Ct_h ,\, w|_{\partial \Omega} = g\} $ is not empty. Hence, we 
can introduce the Oswald interpolation $\Ci^{os}_h:W_h \rightarrow W^c_{h,g}$. Given a function $v_h \in W_h$, the operator $\Ci^{os}_h$ is prescribed at the 
Lagrangian nodes in the interior of $\Omega$ by the average of the values of $v_h$ at this node. For the nodes at the boundary $\partial \Omega$, $\Ci^{os}_h$ 
is prescribed at the Lagrangian nodes on $\partial \Omega$ by the value of $g$ at this node. The following estimate has been analyzed for nonconforming mesh 
and conforming mesh in \cite{OP03,OP07} and was extended to variable polynomial degree in \cite{ZGHS11}.
 
\begin{lemma}\label{os_est_lemma}
{\rm (cf. \cite{OP03,OP07,ZGHS11,Cockburn13})} 
If the Dirichlet data $g$ is contained in $C(\partial \Omega)\cap M_{h}|_{\partial \Omega}$, then 
for any $v_h \in W_h$ and any multi-index $\alpha$ with $|\alpha|=0,1$, the following estimate holds: 
\begin{align}
& \sum_{T \in \Ct_h} \| D^\alpha ( v_h -\Ci^{os}_h  v_h )\|^2_{0,T}\\
\nonumber
\leq & C_3 \left(\sum_{F \in \Ce^0_h} h^{1-2|\alpha|}_F \| \llbracket v_h\rrbracket \|^2_{0,F} 
+\sum_{F \in \Ce^{\partial }_h} h^{1-2|\alpha|}_F \|  g-v_h \|^2_{0,F}\right).
\end{align}
\end{lemma}

In order to prove the local efficiency of the estimators, proper element and face bubble functions are useful. We set the element bubble function in the element $T$ as $B_T = \prod_{i=1}^{d+1} \lambda_i$, where $\lambda_i$ denotes the linear nodal basis function at $i$th vertex in $T$. Besides the element bubble function, as mention in Lemma 3.3 in \cite{Verfurth98} and Lemma 3.6 in \cite{Verfurth05}, there also exists proper face bubble function $B_F$ with $F \in \Ce^0_h$ such that the following lemma holds.
\begin{lemma}\label{bubble_lemma}
{\rm (cf. \cite{Verfurth98,Verfurth05})} For any element $T \in \Ct_h$, a polynomial $\phi \in \Cp_p(T)$ and any face $F\in \Ce^0_h$, a polynomial $\psi \in \Cp_p(F)$, the following estimates hold:
\begin{align*}
\| \phi \|^2_{0,T} &\leq C_4 (\phi,B_T\phi)_{T},\\
\|B_T\phi\|_{0,T} &\leq C_5 \|\phi\|_{0,T} ,\\
\| \psi \|^2_{0,F} &\leq C_6 \langle \psi,B_F\psi\rangle_{F},\\
\|B_F \psi \|_{0,\omega_F} &\leq C_7 \epsilon^{\frac{1}{4}} \alpha^{\frac{1}{2}}_F \| \psi \|_{0,F},\\
\|B_F \psi \|_{0,\omega_F} + \epsilon^{\frac{1}{2}} \| \nabla B_F \psi \|_{0,\omega_F} &\leq C_8\epsilon^{\frac{1}{4}} \alpha^{-\frac{1}{2}}_F \| \psi \|_{0,F}.
\end{align*}
\end{lemma}

\section{Proof of reliability}
In this section, we give the proof of Theorem~\ref{thm_reliability}, which shows reliability of the a posteriori error estimator in Definition~\ref{def_estimator}. 


In view of the assumption (A4), we define a weighted function 
\begin{align}
\varphi := e^{-\psi} +\chi,  \label{varphi}
\end{align}
where $\chi$ is a positive constant to be determined later. Let $\Be_\Bq = \Bq-\Bq_h,\Be_u = u-u_h$. 
We have the following Lemma~\ref{pre-lemma}.
\begin{lemma}\label{pre-lemma}
Let $\varphi$ be given in (\ref{varphi}) with $\chi \geq 2b_0\| e^{-\psi} \|_{L^\infty(\Omega)} \| \nabla \psi\|^2_{L^\infty(\Omega)}$. 
Then the following estimate holds:
\begin{align}
\label{pre-est}
& C \left(\epsilon^{-1}\Vert \Be_\Bq \Vert_{\mathcal{T}_{h}}^{2} + \Vert \Be_u\Vert_{\mathcal{T}_{h}}^{2}\right)\\
\nonumber 
\leq  & \epsilon^{-1}( \Be_\Bq, \varphi \Be_\Bq )_{\Ct_h} -( \Be_u, \nabla \varphi \cdot \Be_{\Bq} )_{\Ct_h} \\
\nonumber
& \quad - \frac{1}{2} ( \bbeta \cdot \nabla \varphi \Be_u,\Be_u )_{\Ct_h}+ \big( (c-\frac{1}{2}{\rm div}\bbeta) \Be_u ,\varphi\Be_u \big)_{\Ct_h}. 
\end{align}
\end{lemma}

\begin{proof}
According to the assumptions (A3)-(A4), we have
\begin{align*}
& \epsilon^{-1}( \Be_\Bq, \varphi \Be_\Bq )_{\Ct_h} +( \Be_u,e^{-\psi} \nabla \psi \cdot \Be_{\Bq} )_{\Ct_h} \nn\\
 & \quad + \frac{1}{2} ( \bbeta \cdot \nabla \psi e^{-\psi} \Be_u,\Be_u )_{\Ct_h}+ \big( (c-\frac{1}{2}{\rm div}\bbeta) \Be_u ,\varphi\Be_u \big)_{\Ct_h} \\
\geq & \epsilon^{-1}\chi( \Be_\Bq, \Be_\Bq )_{\Ct_h} +( \Be_u,e^{-\psi} \nabla \psi \cdot \Be_{\Bq} )_{\Ct_h}+ \frac{b_0}{2} ( e^{-\psi} \Be_u,\Be_u )_{\Ct_h}.
\end {align*}
By the Cauchy-Schwarz and Young's inequalities, for any $\delta>0$, we have
\begin{align*}
&( \Be_u,e^{-\psi} \nabla \psi \cdot \Be_{\Bq} )_{\Ct_h} \\
\leq & \frac{1}{2}\left(  \delta^{-1} \| \nabla \psi \|^2_{L^\infty(\Omega)} \| e^{-\psi} \|_{L^\infty(\Omega)} (\Be_{\Bq},\Be_{\Bq})_{\Ct_h}
 + \delta\| e^{-\psi} \|_{L^\infty(\Omega)}(\Be_u,\Be_u)_{\Ct_h} \right).
\end{align*}
Then,  by taking $\chi \geq 2b_0\| e^{-\psi} \|_{L^\infty(\Omega)} \| \nabla \psi\|^2_{L^\infty(\Omega)} $ and $\delta = \frac{b_0}{2}$, 
we can conclude that the proof is complete. 
\end{proof}

We are now ready to state a key result of the upper bound estimate of 
$\left(\epsilon^{-1}\Vert \Be_\Bq \Vert_{\mathcal{T}_{h}}^{2} + \Vert \Be_u\Vert_{\mathcal{T}_{h}}^{2}\right)$.
\begin{lemma}\label{pre-reliability}
If the Dirichlet data $g$ is contained in $C(\partial \Omega)\cap M_{h}|_{\partial \Omega}$, then 
\begin{align}
& \left(\epsilon^{-1}\Vert \Be_\Bq \Vert_{\mathcal{T}_{h}}^{2} + \Vert \Be_u\Vert_{\mathcal{T}_{h}}^{2}\right)
\leq C \Big( \sum_{T\in \Ct_h}\eta^2_T + \sum_{F\in \Ce^0_h}(\eta^0_F)^{2} + \sum_{F\in \Ce^\partial_h} (\eta^\partial_F)^2  \Big).
\end{align}
\end{lemma}
\begin{proof}
According to Lemma \ref{pre-lemma}, we have
\begin{align}
\label{est_1}
& C \left(\epsilon^{-1}\Vert \Be_\Bq \Vert_{\mathcal{T}_{h}}^{2} + \Vert \Be_u\Vert_{\mathcal{T}_{h}}^{2}\right)\\
\nonumber 
\leq  & \epsilon^{-1}( \Be_\Bq, \varphi \Be_\Bq )_{\Ct_h} -( \Be_u, \nabla \varphi \cdot \Be_{\Bq} )_{\Ct_h} \\
\nonumber
& \quad - \frac{1}{2} ( \bbeta \cdot \nabla \varphi \Be_u,\Be_u )_{\Ct_h}+ \big( (c-\frac{1}{2}{\rm div}\bbeta) \Be_u ,\varphi\Be_u \big)_{\Ct_h}. 
\end{align}
Let $u^{\Ci}_h:=\Ci^{os}_hu_h$. By the definition of $\Ci^{os}_h$, clearly, we have $u^{\Ci}_h|_{\partial \Omega} = g$. 
Adding and subtracting $\epsilon\nabla u^{\Ci}_h$ into $\Bq-\Bq_h$, we have
\begin{align}
\label{est_2}
& \epsilon^{-1}( \Be_\Bq, \varphi \Be_\Bq )_{\Ct_h} \\
\nn
= & -\left( \Bq-\Bq_h,\varphi( \nabla u - \nabla u^{\Ci}_h) \right)_{\Ct_h} - \epsilon^{-1}\left( \Bq-\Bq_h,\varphi\,( \epsilon \nabla u^{\Ci}_h+\Bq_h)  \right)_{\Ct_h}\nn\\
= & \left( \nabla \varphi \cdot(\Bq-\Bq_h),u-u^{\Ci}_h\right)_{\Ct_h} + \left( \varphi\, {\rm div}(\Bq-\Bq_h),u-u^{\Ci}_h  \right)_{\Ct_h} \nn\\
&\qquad - \langle (\Bq-\Bq_h)\cdot \Bn, \varphi(u-u^{\Ci}_h)\rangle_{\partial \Ct_h}-\epsilon^{-1}\left( \Bq-\Bq_h,\varphi\,( \epsilon \nabla u^{\Ci}_h+\Bq_h)  \right)_{\Ct_h}\nn\\
= & \left( \nabla \varphi \cdot(\Bq-\Bq_h),u-u^{\Ci}_h\right)_{\Ct_h} + \left( R_h,\varphi (u-u^{\Ci}_h )\right)_{\Ct_h} \nn\\
&\qquad +\langle \Bq_h \cdot \Bn, \varphi( u-u^{\Ci}_h )\rangle_{\partial \Ct_h}- \left( \bbeta \cdot \nabla(u-u_h) + c(u-u_h), \varphi( u-u^{\Ci}_h )  \right)_{\Ct_h}\nn\\
\nonumber
& \qquad 
 - \epsilon^{-1}\left( \Bq-\Bq_h,\varphi\,( \epsilon \nabla u^{\Ci}_h+\Bq_h)  \right)_{\Ct_h}.
\end{align}
In the last step of (\ref{est_2}), we have used the equation (\ref{mixed-eqs}) and the fact that 
$\langle \Bq \cdot \Bn, \varphi( u-u^{\Ci}_h )\rangle_{\partial \Ct_h}=0$ 
(since jumps of $\Bq\cdot\Bn, u, u^{\Ci}_h$ vanish on all interior faces and $ u-u^{\Ci}_h = 0$ on $\partial \Omega$). 
Inserting (\ref{est_2}) into (\ref{est_1}) and subtracting and adding $u_h$ into $u-u^{\Ci}_h$ 
in the first term of the right-hand side of (\ref{est_2}), we have
\begin{align}
\label{est_21}
& C\left(\epsilon^{-1}\Vert \Be_\Bq \Vert_{\mathcal{T}_{h}}^{2} + \Vert \Be_u\Vert_{\mathcal{T}_{h}}^{2}\right)\\
\nonumber
\leq & \left( \nabla \varphi \cdot(\Bq-\Bq_h),u_h-u^{\Ci}_h\right)_{\Ct_h} + \left( R_h,\varphi (u-u^{\Ci}_h )\right)_{\Ct_h} \\
\nn
&\quad +\langle \Bq_h \cdot \Bn, \varphi( u-u^{\Ci}_h )\rangle_{\partial \Ct_h}
- \left( \bbeta \cdot \nabla(u-u_h) + c(u-u_h), \varphi( u-u^{\Ci}_h )  \right)_{\Ct_h}\\
\nn
&\quad  -\epsilon^{-1}\left( \Bq-\Bq_h,\varphi\,( \epsilon \nabla u^{\Ci}_h+\Bq_h)  \right)_{\Ct_h}\\
\nn
&\quad  -\frac{1}{2} ( \bbeta \cdot \nabla \varphi \Be_u,\Be_u )_{\Ct_h}+ \big( (c-\frac{1}{2}{\rm div}\bbeta) \Be_u ,\varphi\Be_u \big)_{\Ct_h} .
\end{align}
For any $w\in  W^c_{1,h} \cap H^1_0(\Omega) $, the equation (\ref{P2}) in the HDG method (\ref{PD}) can be rewritten as follows after integration by parts:
\begin{align*}
(f,w)_{\Ct_h}& = -(\Bq_h,\nabla w)_{\Ct_h} + \left({\rm div}(\bbeta u_h),w\right)_{\Ct_h} - \langle \bbeta u_h \cdot \Bn,w\rangle_{\partial \Ct_h} \\
&\quad + (cu_h - {\rm div}\bbeta u_h,w)_{\Ct_h} + \langle ( \widehat{\Bq}_h+\widehat{\bbeta u}_h )\cdot\Bn,w\rangle_{\partial \Ct_h}\\
&=({\rm div}\Bq_h,w)_{\Ct_h} - \langle \Bq_h \cdot \Bn,w \rangle_{\partial \Ct_h}  + ( \bbeta \cdot \nabla u_h+cu_h,w )_{\Ct_h}\\
&\quad - \langle \bbeta  \cdot \Bn u_h,w\rangle_{\partial \Ct_h}+ \langle ( \widehat{\Bq}_h+\widehat{\bbeta u}_h )\cdot\Bn,w\rangle_{\partial \Ct_h}.
\end{align*}
Note that the equation (\ref{P4}) indicates $\langle ( \widehat{\Bq}_h+\widehat{\bbeta u}_h )\cdot\Bn,w\rangle_{\partial \Ct_h} = 0$. 
Hence, we have
\begin{align}
- \langle ({\Bq}_h+\bbeta {u}_h )\cdot\Bn,w\rangle_{\partial \Ct_h}  = (R_h,w)_{\Ct_h}. \label{est_3}
\end{align}

Combining (\ref{est_21}) and (\ref{est_3}), we obtain
\begin{align*}
& C \left(\epsilon^{-1}\Vert \Be_\Bq \Vert_{\mathcal{T}_{h}}^{2} 
+\Vert \Be_u\Vert_{\mathcal{T}_{h}}^{2}\right)  \leq \sum^4_{l=1}I_l \text{ where }\\
I_1& =  -\left(  \nabla \psi e^{-\psi}(\Bq-\Bq_h),u_h-u^{\Ci}_h\right)_{\Ct_h} -\epsilon^{-1}\left(  \Bq-\Bq_h,\varphi (\Bq_h+\epsilon \nabla u^{\Ci}_h)\right)_{\Ct_h},\\
I_2&= \left( R_h, (I-\pi_h)( \varphi u-\varphi u^{\Ci}_h )\right)_{\Ct_h} +\langle (\Bq_h+\bbeta u_h)\cdot\Bn, (I-\pi_h)( \varphi u-\varphi u^{\Ci}_h )\rangle_{\partial \Ct_h} ,\\
I_3 & = -\langle \bbeta \cdot \Bn u_h, \varphi( u-u^{\Ci}_h )\rangle_{\partial \Ct_h} -\left( \bbeta \cdot \nabla(u^{\Ci}_h-u_h) + c(u^{\Ci}_h-u_h), 
\varphi( u-u^{\Ci}_h )  \right)_{\Ct_h},\\
I_4 & = -\left( \bbeta \cdot \nabla(u-u^{\Ci}_h) + c(u-u^{\Ci}_h), \varphi( u-u^{\Ci}_h )  \right)_{\Ct_h} \\
&\qquad -  \frac{1}{2} ( \bbeta \cdot \nabla \varphi \Be_u,\Be_u )_{\Ct_h}+ \big( (c-\frac{1}{2}{\rm div}\bbeta) \Be_u ,\varphi\Be_u \big)_{\Ct_h}.
\end{align*}
Here $\pi_h$ is the Cl\'{e}ment-type interpolation into the space $W^c_{1,h}\cap H^1_0(\Omega)$ (see Definition \ref{define-clement}).

Now we estimate the summation $\sum^4_{l=1}I_l$. For the sake of simplicity, given any $v \in H^1(D)$, $D\subset \Omega$, we define an energy norm for $v$ by $\interleave v \interleave_D = \left(\|v\|^2_{0,D} + \epsilon \| \nabla v\|^2_{0,D}\right)^{\frac{1}{2}}$. We can refer to Appendix A and Appendix B for the estimates for $I_{1}$ and $I_{4}$ respectively. In the following we mainly focus on the estimate of $I_2+I_3$ by two approaches. 

(\textbf{Approach A})
For the first approach, we consider the estimate of $I_2$ and $I_3$ separately. By the approximation properties of the Cl\'{e}ment-type 
interpolation presented in Lemma \ref{ax_re_1}, we obtain
\begin{align*}
I_2 &\leq C\sum_{T \in \Ct_h} \alpha_T \| R_h\|_{0,T} \interleave \varphi u-\varphi u^{\Ci}_h\interleave_{\Omega_T}   \\
& \quad +C \sum_{F\in \Ce^0_h} \left(\epsilon^{-\frac{1}{4}} \alpha^{\frac{1}{2}}_F  \| \llbracket \Bq_h\cdot\Bn \rrbracket   \|_{0,F} + \epsilon^{-\frac{1}{4}} \alpha^{\frac{1}{2}}_F \|\bbeta\|_{L^\infty(F)} \| \llbracket u_h \rrbracket \|_{0,F}\right)  \interleave \varphi u-\varphi u^{\Ci}_h\interleave_{\Omega_F}  .
\end{align*}
Using the Young's inequality and subtracting and adding $u_h$ into $u-u^{\Ci}_h$, we further have
\begin{align}
\label{est_I2}
I_2 & \leq \frac{C}{2\delta} \sum_{T \in \Ct_h} \alpha^2_T \|R_h\|^2_{0,T}+  \frac{C}{2\delta} \sum_{F \in \Ce^0_h} \epsilon^{-\frac{1}{2}} \alpha_F  \| \llbracket \Bq_h\cdot\Bn \rrbracket \|^2_{0,F}\\
\nn
&\quad +\frac{C}{2\delta} \sum_{F \in \Ce^0_h} \epsilon^{-\frac{1}{2}} \alpha_F \|\bbeta\|_{L^\infty(F)} \| \llbracket  u_h \rrbracket   \|^2_{0,F}+ \delta C  \|\bbeta\|_{L^\infty(\Omega)} \interleave \varphi(u-u^{\Ci}_h)\interleave^2_{\Ct_h}. 
\end{align}
Since $u-u^{\Ci}_h=0$ on $\partial \Omega$, we easily get $\langle \bbeta \cdot \Bn u^{\Ci}_h, \varphi( u-u^{\Ci}_h )\rangle_{\partial \Ct_h}=0$. 
Via integrating by parts, we have
\begin{align*}
I_3 &= \left( \bbeta (u^{\Ci}_h-u_h), \varphi \nabla(u-u^{\Ci}_h) \right)_{\Ct_h} + \left( \bbeta \cdot \nabla \varphi (u^{\Ci}_h-u_h), u-u^{\Ci}_h \right)_{\Ct_h} \\
&\quad +\left( ({\rm div}\bbeta -c)(u^{\Ci}_h-u_h), \varphi(u-u^{\Ci}_h) \right)_{\Ct_h}.
\end{align*} 
Note that $\nabla(u-u^{\Ci}_h) = -\epsilon^{-1}\{ (\Bq-\Bq_h) +(\Bq_h+\epsilon \nabla u_h) -\epsilon \nabla(u_h - u^{\Ci}_h) \}$ and $u-u^{\Ci}_h = (u-u_h) + (u_h-u^{\Ci}_h)$, we utilize the Cauchy-Schwarz and Young's inequalities to obtain
\begin{align}
\label{est_I3}
I_3 &\leq \big(C^{d}_1 \frac{3\epsilon^{-1}}{2\delta}+ C^d_2(1+\frac{1}{2\delta})  \big) \| u_h - u^{\Ci}_h \|^2_{0,\Ct_h} +C^d_2 \frac{\delta}{2} \|u-u_h\|^2_{0,\Ct_h}\\
\nn
&\quad +\frac{\delta}{2} \left( \epsilon^{-1}\|\Bq-\Bq_h\|^2_{0,\Ct_h} + \epsilon^{-1}\| \Bq_h + \epsilon \nabla u_h \|^2_{0,\Ct_h} + \epsilon\|\nabla(u_h - u^{\Ci}_h)\|^2_{0,\Ct_h}  \right),
\end{align}
where $C^d_1 = \|\bbeta\varphi\|^2_{L^\infty(\Omega)}, C^d_2=\| \bbeta \cdot\nabla \varphi+\varphi({\rm div}\bbeta -c) \|_{L^\infty(\Omega)}$. 
So, by the first approach,
\begin{align}
\label{I2_I3_1}
& I_{2} + I_{3}\\
\nn
\leq & \frac{C}{2\delta} \sum_{T \in \Ct_h} \alpha^2_T \|R_h\|^2_{0,T}+  \frac{C}{2\delta} \sum_{F \in \Ce^0_h} \epsilon^{-\frac{1}{2}} \alpha_F  \| \llbracket \Bq_h\cdot\Bn \rrbracket \|^2_{0,F}\\
\nn
&\quad +\frac{C}{2\delta} \sum_{F \in \Ce^0_h} \epsilon^{-\frac{1}{2}} \alpha_F \|\bbeta\|_{L^\infty(F)} \| \llbracket  u_h \rrbracket   \|^2_{0,F}+ \delta C  \|\bbeta\|_{L^\infty(\Omega)} \interleave \varphi(u-u^{\Ci}_h)\interleave^2_{\Ct_h}\\
\nn
&\quad +\big(C^{d}_1 \frac{3\epsilon^{-1}}{2\delta}+ C^d_2(1+\frac{1}{2\delta})  \big) \| u_h - u^{\Ci}_h \|^2_{0,\Ct_h} 
+C^d_2 \frac{\delta}{2} \|u-u_h\|^2_{0,\Ct_h}\\
\nn
&\quad +\frac{\delta}{2} \left( \epsilon^{-1}\|\Bq-\Bq_h\|^2_{0,\Ct_h} + \epsilon^{-1}\| \Bq_h + \epsilon \nabla u_h \|^2_{0,\Ct_h} 
+ \epsilon\|\nabla(u_h - u^{\Ci}_h)\|^2_{0,\Ct_h}  \right).
\end{align}
Here,  we recall that $C^d_1 = \|\bbeta\varphi\|^2_{L^\infty(\Omega)}$ and $C^d_2=\| \bbeta 
\cdot\nabla \varphi+\varphi({\rm div}\bbeta -c) \|_{L^\infty(\Omega)}$.

(\textbf{Approach B})
For the second approach, we estimate the summation of $I_2$ and $I_3$. It is clear that
\begin{align}
\label{I2_I3_eq1}
& I_2+ I_3\\
\nn
= &\left( R_h, (I-\pi_h)( \varphi u-\varphi u^{\Ci}_h )\right)_{\Ct_h} +\langle \Bq_h\cdot\Bn, (I-\pi_h)( \varphi u-\varphi u^{\Ci}_h )\rangle_{\partial \Ct_h}\\
\nn
& \quad  -\langle \bbeta \cdot \Bn u_h,\pi_h( \varphi u- \varphi u^{\Ci}_h )\rangle_{\partial \Ct_h}\\
\nn
&\quad  -\left( \bbeta \cdot \nabla(u^{\Ci}_h-u_h) + c(u^{\Ci}_h-u_h), \varphi( u-u^{\Ci}_h )  \right)_{\Ct_h}.
\end{align}
For the first two terms of the right-hand side of (\ref{I2_I3_eq1}), the estimates can be similarly obtained as in (\ref{est_I2}). 
For the third term, by the trace inequality we have
\begin{align}
& -\langle \bbeta \cdot \Bn u_h,\pi_h( \varphi u- \varphi u^{\Ci}_h )\rangle_{\partial \Ct_h} \nn\\
\leq & C\sum_{F \in \Ce^0_h} \|\bbeta\|_{L^\infty(F)} \| \llbracket u_h\rrbracket  \|_{0,F} h^{-\frac{1}{2}}_F \| \pi_h( \varphi u - \varphi u^{\Ci}_h )  \|_{0,T_F}\nn\\
  \leq  & C\sum_{F \in \Ce^0_h} \|\bbeta\|_{L^\infty(F)} \| \llbracket u_h\rrbracket  \|_{0,F} h^{-\frac{1}{2}}_F \|  \varphi u - \varphi u^{\Ci}_h   \|_{0,\Omega_{T_F}} \nn\\ 
 \leq & \frac{C}{2 \delta} \sum_{F \in \Ce^0_h} \|\bbeta\|_{L^\infty(F)} h^{-1}_F\|    \llbracket u_h\rrbracket  \|^2_{0,F}  +C \delta \|\bbeta\|_{L^\infty(\Omega)}  \| \varphi \|^2_{L^\infty(\Omega)} ( \|u-u_h\|^2_{0,\Ct_h} +  \|u_h-u^{\Ci}_h\|^2_{0,\Ct_h} ),\nn
\end{align}
where $T_F$ is an element which satisfies $F \subset \partial T$ and the second inequality of the above estimate is deduced from the $L^2$ stability property of Cl\'{e}ment-type interpolation $\pi_h$ (cf. \cite{Verfurth98}). For the fourth term of the right-hand side of (\ref{I2_I3_eq1}), 
we can easily derive, for any $\delta > 0$, that
\begin{align}
&-\left( \bbeta \cdot \nabla(u^{\Ci}_h-u_h) + c(u^{\Ci}_h-u_h), \varphi( u-u^{\Ci}_h )  \right)_{\Ct_h}\nn\\
\leq & \frac{1}{\delta} \| \bbeta \varphi \|_{L^\infty(\Omega)} \|\nabla (u_h -  u^{\Ci}_h)\|^2_{0,\Ct_h} + C^d_3  \|u_h-u^{\Ci}_h\|^2_{0,\Ct_h}+  \frac{ \delta}{2}C^d_4 \| u-u_h \|^2_{0,\Ct_h},\nn 
\end{align}
where $C^d_3 = \frac{\delta}{2} \|\bbeta \varphi\|_{L^\infty(\Omega)} + (\frac{1}{2\delta}+1)\|c\varphi\|_{L^\infty(\Omega)} $, $C^d_4= \|\bbeta \varphi\|_{L^\infty(\Omega)} +\|c\varphi\|_{L^\infty(\Omega)} $. Thus, by the second approach, 
\begin{align}
\label{I2I3_r}
& I_2 + I_3\\
\nn
\leq &\frac{C}{2\delta} \sum_{T \in \Ct_h} \alpha^2_T \|R_h\|^2_{0,T}+  \frac{C}{2\delta} \sum_{F \in \Ce^0_h} \epsilon^{-\frac{1}{2}} \alpha_F  
\| \llbracket \Bq_h  \cdot\Bn\rrbracket \|^2_{0,F} + \delta C\interleave \varphi(u-u^{\Ci}_h)\interleave^2_{\Ct_h}\\
\nn
& +  \frac{C}{2 \delta} \sum_{F \in \Ce^0_h} \|\bbeta\|_{L^\infty(F)} h^{-1}_F\|    \llbracket u_h\rrbracket  \|^2_{0,F}  +C \delta \|\bbeta\|_{L^\infty(\Omega)} \| \varphi \|^2_{L^\infty(\Omega)} ( \|u-u_h\|^2_{0,\Ct_h} +  \|u_h-u^{\Ci}_h\|^2_{0,\Ct_h} )\\
\nn
& +  \frac{1}{\delta} \| \bbeta \varphi \|_{L^\infty(\Omega)} \|\nabla (u_h -  u^{\Ci}_h)\|^2_{0,\Ct_h} + C^d_3  \|u_h-u^{\Ci}_h\|^2_{0,\Ct_h}+  \frac{ \delta}{2}C^d_4 \| u-u_h \|^2_{0,\Ct_h}.
\end{align}

For the term $\delta C\interleave \varphi(u-u^{\Ci}_h)\interleave^2_{\Ct_h} $ in the right-hand sides of (\ref{I2_I3_1}) and (\ref{I2I3_r}), we derive that
\begin{align}
\label{I2I3_rr}
& \delta C\interleave \varphi(u-u^{\Ci}_h)\interleave^2_{\Ct_h}\\
\nn
\leq & \delta C \left(  \epsilon\| \nabla \varphi\|^2_{L^\infty(\Omega)} +\|\varphi\|^2_{L^\infty(\Omega)}  \right)\left(  \|u-u_h\|^2_{0,\Ct_h} + \|u_h-u^{\Ci}_h\|^2_{0,\Ct_h}  \right)  \\ 
\nn
& + \delta C \|\varphi\|^2_{L^\infty(\Omega)} \left( \epsilon^{-1}\|\Bq-\Bq_h\|^2_{0,\Ct_h} + \epsilon^{-1}\| \Bq_h + \epsilon \nabla u_h \|^2_{0,\Ct_h} + \epsilon\|\nabla(u_h - u^{\Ci}_h)\|^2_{0,\Ct_h} \right).
\end{align}

Now, we are ready to finish the proof.  
Combining (\ref{est_I1}), (\ref{I2_I3_1}) (which is given by the first approach for the estimate of $I_2+I_3$), 
(\ref{I2I3_rr}), (\ref{est_I4}) and Lemma~\ref{os_est_lemma}, and choosing $\delta$ small enough, we have
\begin{align}
\label{est-f1}
\left(\epsilon^{-1}\Vert \Be_\Bq \Vert_{\mathcal{T}_{h}}^{2} + \Vert \Be_u\Vert_{\mathcal{T}_{h}}^{2}\right) \leq C \big( \sum_{T\in \Ct_h}\eta^2_T 
+ \sum_{F\in \Ce^0_h}(\eta^{0}_{F,1})^{2} + \sum_{F\in \Ce^\partial_h} (\eta^\partial_{F,1})^2  \big),
\end{align}
where $
\eta_{F,1}^{0} = \left(  \epsilon^{-\frac{1}{2}} \alpha_F \| \llbracket \Bq_h\cdot\Bn \rrbracket  \|^2_{0,F} + \gamma_{F,1}\| \llbracket  u_h \rrbracket  \|^2_{0,F} \right)^{\frac{1}{2}}$, $\eta^{\partial}_{F,1} =    \gamma_{F,1}^{\frac{1}{2}} \|g-u_h \|_{0,F}$ and $\gamma_{F,1} =\frac{\epsilon}{h_F}+\frac{h_F}{\epsilon}+\epsilon^{-\frac{1}{2}}\alpha_F $. Similarly, we can obtain the following estimate by combining (\ref{est_I1}), (\ref{I2I3_r}) 
(which is given by the first approach for the estimate of $I_2+I_3$), (\ref{I2I3_rr}), (\ref{est_I4}) and Lemma~\ref{os_est_lemma}, and again 
choosing $\delta$ small enough,
\begin{align}
\label{est-f2}
\left(\epsilon^{-1}\Vert \Be_\Bq \Vert_{\mathcal{T}_{h}}^{2} + \Vert \Be_u\Vert_{\mathcal{T}_{h}}^{2}\right)
\leq C \big( \sum_{T\in \Ct_h}\eta^2_T + \sum_{F\in \Ce^0_h}(\eta^{0}_{F,2})^{2} + \sum_{F\in \Ce^\partial_h} (\eta^\partial_{F,2})^2  \big),
\end{align}
where $
\eta_{F,2}^{0} = \left(  \epsilon^{-\frac{1}{2}} \alpha_F \| \llbracket \Bq_h\cdot\Bn \rrbracket  \|^2_{0,F} + \gamma_{F,2}\| \llbracket  u_h \rrbracket  \|^2_{0,F} \right)^{\frac{1}{2}}$, $\eta^{\partial}_{F,2} =    \gamma_{F,2}^{\frac{1}{2}} \|g-u_h \|_{0,F}$ and $\gamma_{F,2} =\frac{ \epsilon + \|\boldsymbol{\beta}\|_{L^\infty(F)}  }{h_F}  + h_F$. 
Thus, combining (\ref{est-f1}) and (\ref{est-f2}) completes the proof.
\end{proof}

Now we are in a position to show the proof of the first main result. 

\begin{proof}(Proof of Theorem~\ref{thm_reliability})
Note that 
$
{\rm div}(\Bq-\Bq_h)  + \bbeta \cdot \nabla(u-u_h) = R_h - c(u-u_h)
$
and the fact $\alpha_T \leq 1$, one can easily obtain the following estimate by triangle inequality,
\begin{align}
\label{pre-est-2}
 \alpha^2_T\| {\rm div}(\Bq-\Bq_h)  + \bbeta \cdot \nabla(u-u_h)  \|^2_{0,T} \leq 2 \alpha^2_T\| R_h \|^2_{0,T} + 2 \|c( u-u_h )\|^2_{0,T}.
\end{align}
Moreover, for the term $\epsilon \|\nabla (u-u_h)\|^2_{0,\Ct_h} $, we have
\begin{align}
\label{pre-est-3}
\epsilon \|\nabla (u-u_h)\|^2_{0,\Ct_h} &= \epsilon^{-1} \|  \Bq - \Bq_h + \Bq_h +\epsilon \nabla u_h \|^2_{0,\Ct_h}\\
\nn
& \leq 2  \epsilon^{-1} \|  \Bq - \Bq_h \|^2_{0,\Ct_h} + 2 \epsilon^{-1} \| \Bq_h +\epsilon \nabla u_h \|^2_{0,\Ct_h} .
\end{align}
Combining (\ref{pre-est-2}), (\ref{pre-est-3}), Lemma~\ref{pre-reliability} and the fact that the jumps of $\Bq\cdot\Bn$ and $u$ vanish on all interior faces, 
we immediately have the following reliability estimate:
\begin{align}
\interleave ( \Bq-\Bq_h,u-u_h ) \interleave^2_h \leq C \Big( \sum_{T\in \Ct_h}\eta^2_T + \sum_{F\in \Ce^0_h}(\eta^{0}_F)^{2} + \sum_{F\in \Ce^\partial_h} (\eta^\partial_F)^2  \Big).\nn
\end{align}
So, the proof is complete.
\end{proof}

\section{Proof of efficiency}
In this section, we give the proofs of Efficiency (Theorem~\ref{thm_efficiency}) and 
Efficiency on refined element (Theorem~\ref{thm_refined_efficiency}), which 
address the efficiency of a posteriori error estimator in Definition~\ref{def_estimator}.

\subsection{Efficiency}
By using the element bubble function $B_T$ (cf. Lemma \ref{bubble_lemma}), we have the following estimate which is proven in Appendix C.
\begin{lemma}
\label{lemma_control_eta2}
For any $T\in \Ct_h$, we have that
\[
\alpha^2_T \| R_h \|^2_{0,T} \leq C \left(\alpha^2_T \| {\rm div}(\Bq-\Bq_h) +\bbeta \cdot \nabla (u-u_h) \|^2_{0,T} + \|u-u_h\|^2_{0,T} + {osc}^2_h(R_h,T) \right),
\]
where the data oscillation term ${osc}^2_h(R_h,T)$ and the projection $P_{W}$ are introduced in Theorem~\ref{thm_efficiency}. 
\end{lemma}

\begin{remark}
For any $w\in H^1_0(T)$, we also have
\[
(R_h,w)_T = -(\Bq-\Bq_h,\nabla w)_T + ( \bbeta \cdot \nabla(u-u_h) +c(u-u_h)  ,w)_T.
\]
By the similar technique, we can deduce that
\begin{align}
\label{eff-1}
&\alpha^2_T \| R_h \|^2_{0,T}\\
\nn
\leq & C \left(  \epsilon^{-1}\| \Bq-\Bq_h \|^2_{0,T} + \alpha^2_T \| \bbeta \cdot \nabla (u-u_h)+c(u-u_h) \|^2_{0,T}  + osc^2_h(R_h,T)   \right). 
\end{align}
\end{remark}

With Lemma~\ref{lemma_control_eta2}, we are in a position to show the proof of the second main result. 

\begin{proof}(Proof of Theorem~\ref{thm_efficiency})
By (\ref{estimator1}), (\ref{estimator2}) and the fact that $\Bq\cdot \Bn$ and $u$ are continuous across all interior faces, we have 
(\ref{in_face_efficiency}) and (\ref{boundary_face_efficiency}) immediately. 

For any $T\in \Ct_h$, adding and subtracting $\Bq$ into $\Bq_h+\epsilon \nabla u_h$ yields
\begin{align}
\label{control_eta1}
\epsilon^{-1}\| \Bq_h + \epsilon \nabla u_h  \|^2_{0,T} \leq 2 \epsilon^{-1} \|\Bq-\Bq_h\|^2_{0,T} + 2 \epsilon \| \nabla(u-u_h) \|^2_{0,T}.
\end{align}
Combining (\ref{control_eta1}) and Lemma~\ref{lemma_control_eta2}, we can conclude that (\ref{elem_efficiency}) is true.
\end{proof}

\subsection{Efficiency on refined element}
The proof of (\ref{ineq_refined_efficiency2}) in Theorem~\ref{thm_refined_efficiency} can be directly obtained by (\ref{eff-1}) and (\ref{control_eta1}), and the estimate (\ref{ineq_refined_efficiency}) is an immediate consequence of 
the following Lemma~\ref{lemma_refined_qn} and Lemma~\ref{lemma_refined_u}.

\begin{lemma}
\label{lemma_refined_qn}
For any $F\in \Ce^0_h$, if $h_F \leq O(\epsilon)$, we have
\begin{align*}
& \epsilon^{-\frac{1}{2}} \alpha_F \| \llbracket \Bq_h\cdot \Bn \rrbracket \|^2_F\\
\leq & C \sum_{T \in \omega_F} \left(  \epsilon^{-1}\| \Bq-\Bq_h \|^2_{0,T} + \epsilon\| \nabla(u-u_h) \|^2_{0,T} 
+\|u-u_h\|^2_{0,T} + osc^2_h(R_h,T)  \right).
\end{align*}
\end{lemma}
\begin{proof}
Note that for any $w \in H^1_0(\omega_F)$, we have
\begin{align*}
\langle \llbracket \Bq_h\cdot\Bn \rrbracket, w \rangle_F& = \sum_{T \in \omega_F} \langle (\Bq- \Bq_h )\cdot \Bn, w \rangle_{\partial T} = \sum_{T \in \omega_F}  \left( ({\rm div}(\Bq-\Bq_h),w)_T + (\Bq-\Bq_h,\nabla w)_T \right) \\
&=\sum_{T \in \omega_F} \left( ( R_h - \bbeta \cdot \nabla(u-u_h)  - c(u-u_h),w )_T+ (\Bq-\Bq_h,\nabla w)_T \right),
\end{align*}
Applying Lemma \ref{bubble_lemma} with $w = B_F \llbracket \Bq_h\cdot\Bn \rrbracket$, we have
\begin{align*}
&\epsilon^{-\frac{1}{2}} \alpha_F \| \llbracket \Bq_h\cdot \Bn \rrbracket\|^2_F\nn\\
& \leq C \sum_{T \in \omega_F}  \left(  \epsilon^{-1}\| \Bq-\Bq_h \|^2_{0,T}  + \alpha^2_T \| R_h \|^2_{0,T} + \alpha^2_T \| \bbeta \cdot \nabla (u-u_h)+c(u-u_h) \|^2_{0,T}  \right).
\end{align*}
If $h_F \leq O(\epsilon)$, we have $\alpha^2_T \leq O(\epsilon)$ for $T\in \omega_F$. Then, by the above inequality and the estimate (\ref{eff-1}),
we can conclude that the proof is complete.
\end{proof}

By the similar approach as in Lemma 3.4 of \cite{Cockburn14}, we get the following estimate.
\begin{lemma}
\label{lemma_refined_u}
For any $F\in \Ce^0_h$, if $h_F \leq O(\epsilon)$, we have
\[
\gamma_F \|\llbracket u_h \rrbracket    \|^2_F \leq C \sum_{T \in \omega_F} \left(  \epsilon^{-1}\| \Bq-\Bq_h \|^2_{0,T} + \epsilon\| \nabla(u-u_h) \|^2_{0,T}  \right).
\]
\end{lemma}
\begin{proof}
We denote $P_{M_0}$ by the $L^2$ orthogonal projection operator onto the space $M_{0,h}$, where $M_{0,h}:=\{ \mu\in L^2(\mathcal E_h)\, : \,  \mu|_F\in \mathcal P_0(F)\ {\rm for\ all }\ F\in \mathcal E_h\}$. By the equation (\ref{P1}) in the HDG method, we have that, for any $T\in \Ct_h$ and any $\Br $ in the lowest order Raviart-Thomas space $RT_0(T)$,
\[
(\epsilon^{-1} {\Bq}_h, \Br)_T-(u_h, {{\rm div \, {\Br}}})_T+\langle \widehat u_h, {{\Br}\cdot {\Bn}}\rangle_{\partial T}=0.
\]
Since $\widehat u_h$ is single-valued on $F$, then for any $\Br \in H({\rm div}, \omega_F)$, we have
\[
\epsilon^{-1} ( \Bq_h + \epsilon \nabla u_h, \Br )_{\omega_F} = -\sum_{T \in \omega_F} \sum_{e \in \partial T \setminus F} 
\langle \widehat u_h-u_h, \Br \cdot \Bn \rangle_F + \langle \llbracket u_h \rrbracket , \Br \cdot \Bn \rangle_F.
\]
We take $\Br_{0}\in H({\rm div}, \omega_F)$ such that $\Br_{0}|_{T} \in RT_0(T) \text{ for all } T \in \omega_F$, 
$ \int_F \Br_{0} \cdot \Bn = \int_F P_{M_0}\llbracket u_h \rrbracket$ and $\int_e \Br_{0} \cdot \Bn = 0$ for all $e \in \partial \omega_{F}$. Then we obtain
\begin{align*}
\|P_{M_0}\llbracket u_h \rrbracket\|^2_{0,F} &=  \langle \llbracket u_h \rrbracket , \Br_{0} \cdot \Bn \rangle_F = \epsilon^{-1} ( \Bq_h + \epsilon \nabla u_h, 
\Br_{0} )_{\omega_F} \\
& \leq C \sum_{T\in \omega_F} \epsilon^{-1} \| \Bq_h + \epsilon \nabla u_h \|_{0,T} \| \Br_{0}  \|_{0,T}.
\end{align*}
Obviously, $\| \Br_{0} \|_{0,T} \leq C h^{\frac{1}{2}}_F \| \Br_{0} \cdot \Bn \|_{0,F}$ (cf. Lemma A.1 in  \cite{Cockburn14}). Therefore we have
\begin{align}\label{eff-2}
\|P_{M_0}\llbracket u_h \rrbracket\|^2_{0,F} \leq C \frac{h_F}{\epsilon^2} \|  \Bq_h + \epsilon \nabla u_h  \|^2_{0,\omega_F}.
\end{align}
Note that 
\begin{align*}
\| (I-P_{M_0})\llbracket u_h \rrbracket\|^2_{0,F} = \| (I-P_{M_0})\llbracket u- u_h \rrbracket\|^2_{0,F}  \leq 2 \sum_{T\in \omega_F} \| (I-P_{W_0})(u -  u_h|_T)\|^2_{0,F}
\end{align*}
where $P_{W_0}$ is the $L^2$ orthogonal projection operator onto the space of piecewise constant functions on each element. Then the trace theorem and Poincar\'{e}'s inequality indicate that 
\begin{align}
\| (I-P_{M_0})\llbracket u_h \rrbracket\|^2_{0,F} \leq C h_F \| \nabla(u-u_h) \|^2_{0,\omega_F}. \label{eff-3}
\end{align}
Combining (\ref{eff-2}) and (\ref{eff-3}) yields that
\[
\| \llbracket u_h \rrbracket\|^2_{0,F} \leq C \left( h_F \| \nabla(u-u_h) \|^2_{0,\omega_F}  +  \frac{h_F}{\epsilon^2} \|  \Bq_h + \epsilon \nabla u_h  \|^2_{0,\omega_F} \right).
\]
By subtracting and adding $\Bq$ into $\Bq_h + \epsilon \nabla u_h  $ and the fact that $\gamma_F \frac{h_F}{\epsilon} \leq O(1)$ if $h_F \leq O(\epsilon)$, 
we can conclude that the proof is complete.
\end{proof}
\section{Numerical experiments}

In this section, we present numerical results of the adaptive HDG method for two dimensional model problems to show how the meshes are generated adaptively and how the estimators and the errors behave due to the effects from the quantity $\epsilon/\|\boldsymbol{\beta}  \|_{L^\infty(\Omega)}$. The adaptive HDG procedure consists of adaptive loops of the cycle ``SOLVE $\rightarrow $ ESTIMATE $\rightarrow $  MARK $\rightarrow $  REFINE''. In the step SOLVE, we choose the direct method such as the the multifrontal method to solve the discrete system. In the step ESTIMATE, we adopt the reliable and efficient a posteriori error estimators suggested in the above sections. In the step REFINE, we apply the newest vertex bisection algorithm (see \cite{RS07} and the references therein for details). For the step MARK, we use the bulk algorithm which defines a set $\mathcal {M}^F_h$ of marked edges such that 
\[
\sum_{F \in \mathcal {M}^F_h} \left((  \eta^0_F))^2+(\eta^{\partial}_F)^2\right) \geq \theta_1 \sum_{F \in \Ce_h} \left((  \eta^0_F))^2+(\eta^{\partial}_F)^2\right)
\]
and a set $\mathcal {M}^T_h$ of marked triangles such that
\[
\sum_{T \in \mathcal {M}^T_h} \eta_T^2\geq \theta_2 \sum_{T \in \Ct_h} \eta_T^2,
\]
where $\theta_1$ and $\theta_2$ are optional parameters and we use $\theta_1=\theta_2=0.5$ in the following experiments. For brevity, we denote by $\eta^2_1 = \sum_{T \in \Ct_h} \eta_T^2$ and $\eta^2_2 = \sum_{F \in \Ce_h} \left((  \eta^0_F))^2+(\eta^{\partial}_F)^2\right)$.

For any $(\Bp,w) \in \BH^1(\Ct_h) \times H^1(\Ct_h)$, we define an error norm 
$
\| (\Bp,w)\|^2_h 
= \sum_{T\in \Ct_h} ( \epsilon^{-1} \| \Bp \|^2_{0,T} + \|w\|^2_{0,T} ).
$
In the following experiments, the  adaptive HDG method is implemented for piecewise linear (HDG-P1), quadratic (HDG-P2), 
and cubic (HDG-P3) finite element spaces. The figures displaying the convergence history are all plotted in log-log coordinates.


\begin{figure}[htbp]
\begin{center}
\includegraphics[width=6cm,height=5.5cm]{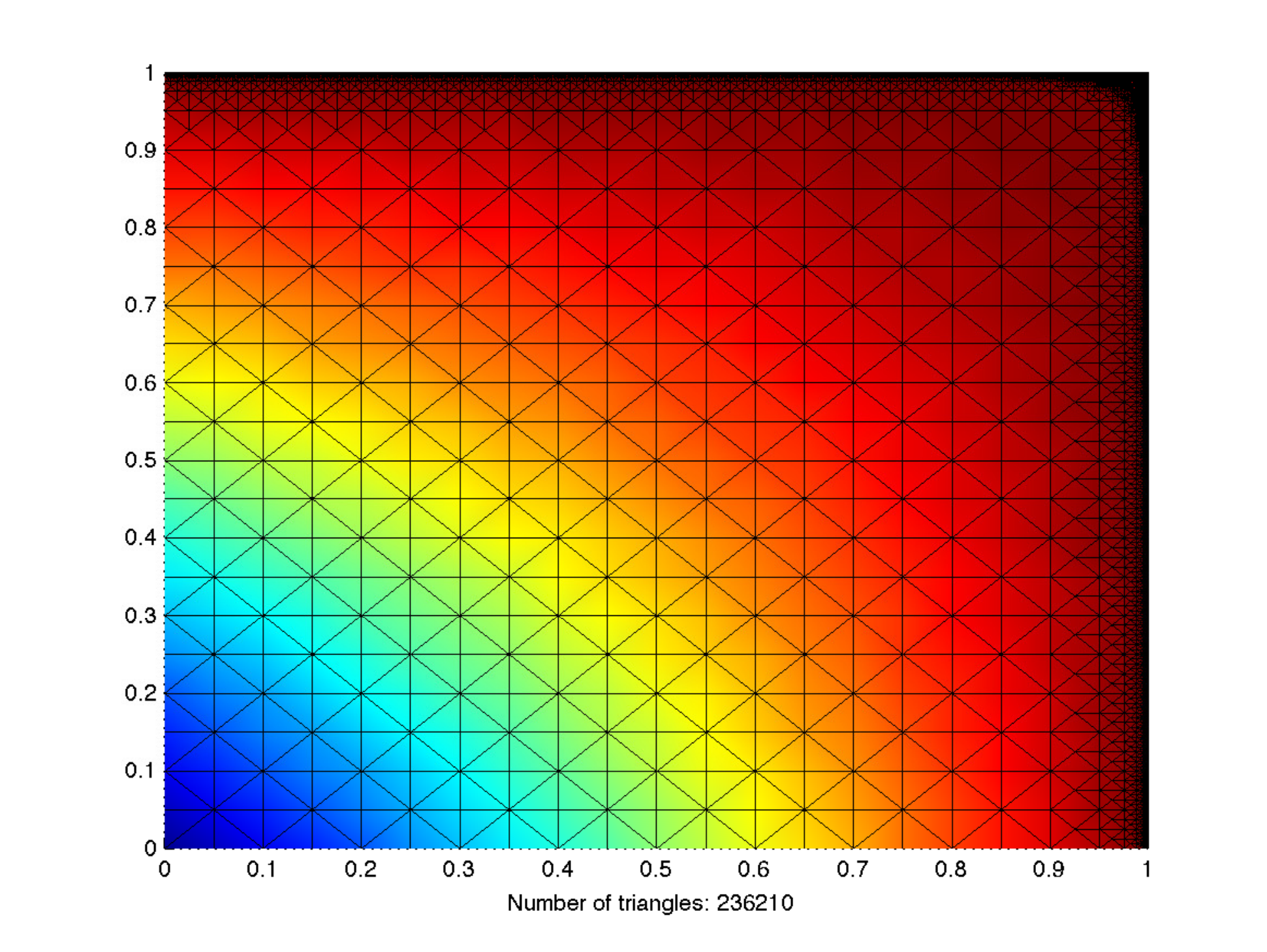}
\includegraphics[width=6cm,height=5.5cm]{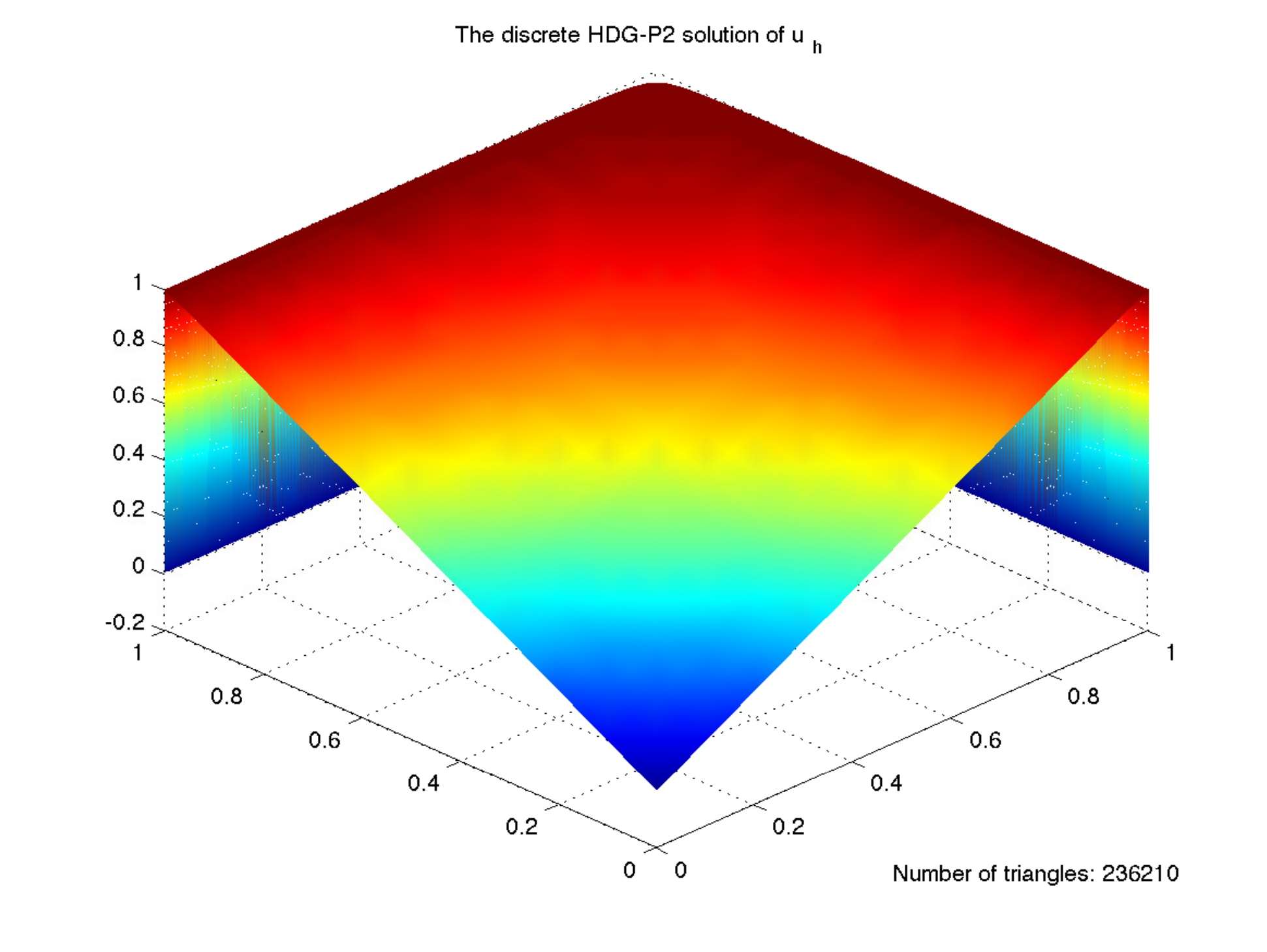}
\end{center}
\vspace{-0.2cm}
\caption{\footnotesize Adaptively refined mesh (left) and 3D plot of the corresponding approximate solution $u_h$ (right)  by HDG-P2 for the case $\epsilon = 10^{-4}$.} \label{ex1-1}
\end{figure}

\begin{exm}
{\rm We consider a boundary layer problem in \cite{AyusoMarini:cdf}. The convection-diffusion equation (\ref{cd_eqs}) is solved in 
the domain $\Omega = [0,1]\times [0,1]$ with $\beta = [1,1]^T$, $c=0$. The source term $f$ and the Dirichlet boundary condition 
are chosen such that 
\[
u(x,y) = x + y(1-x) + \frac{ e^{-1/\epsilon} - e^{-(1-x)(1-y)/\epsilon} }{1-e^{-1/\epsilon}}
\]
is the exact solution.

\begin{figure}[htbp]
\begin{center}
\includegraphics[width=6cm,height=5.5cm]{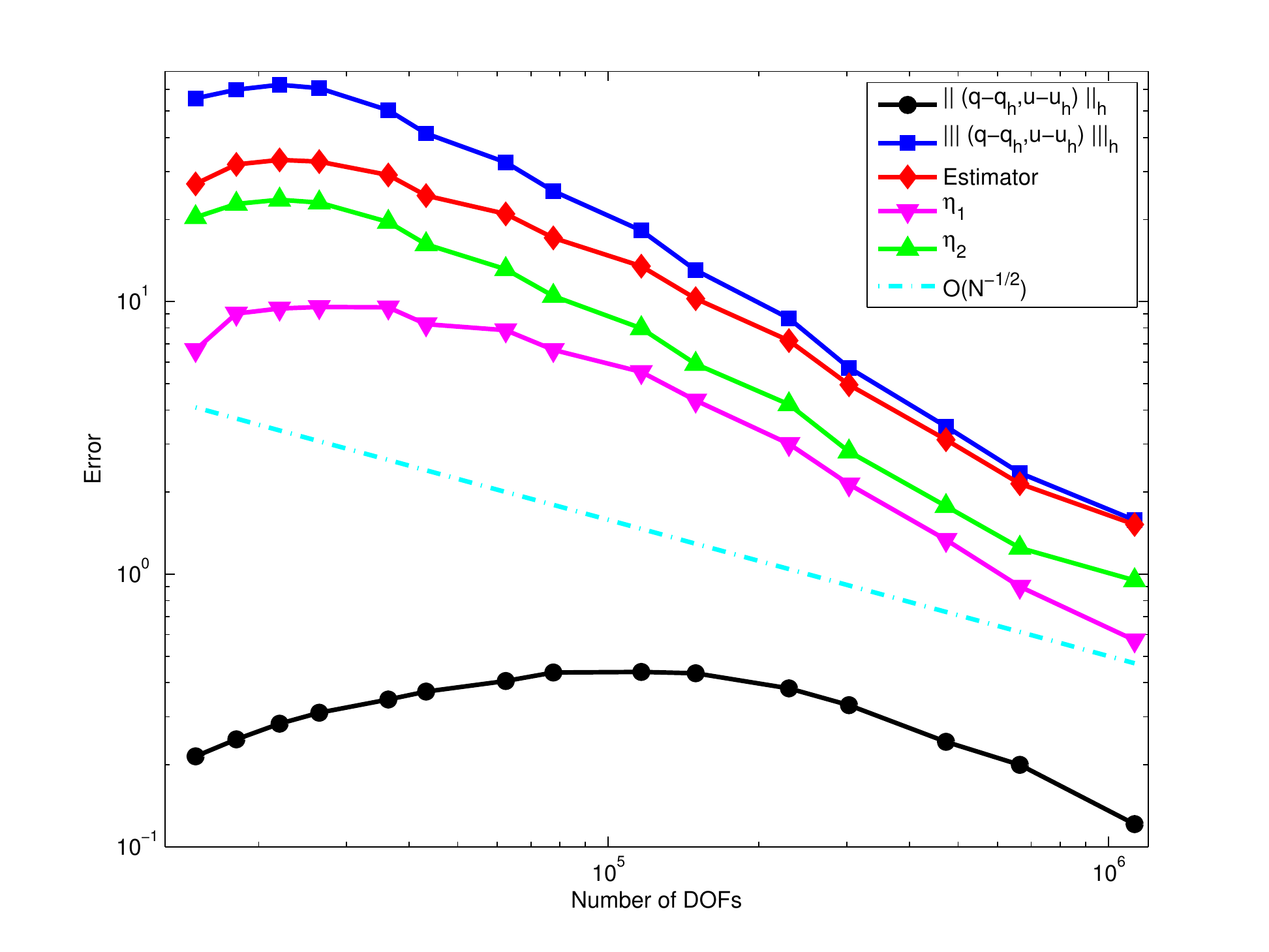}
\includegraphics[width=6cm,height=5.5cm]{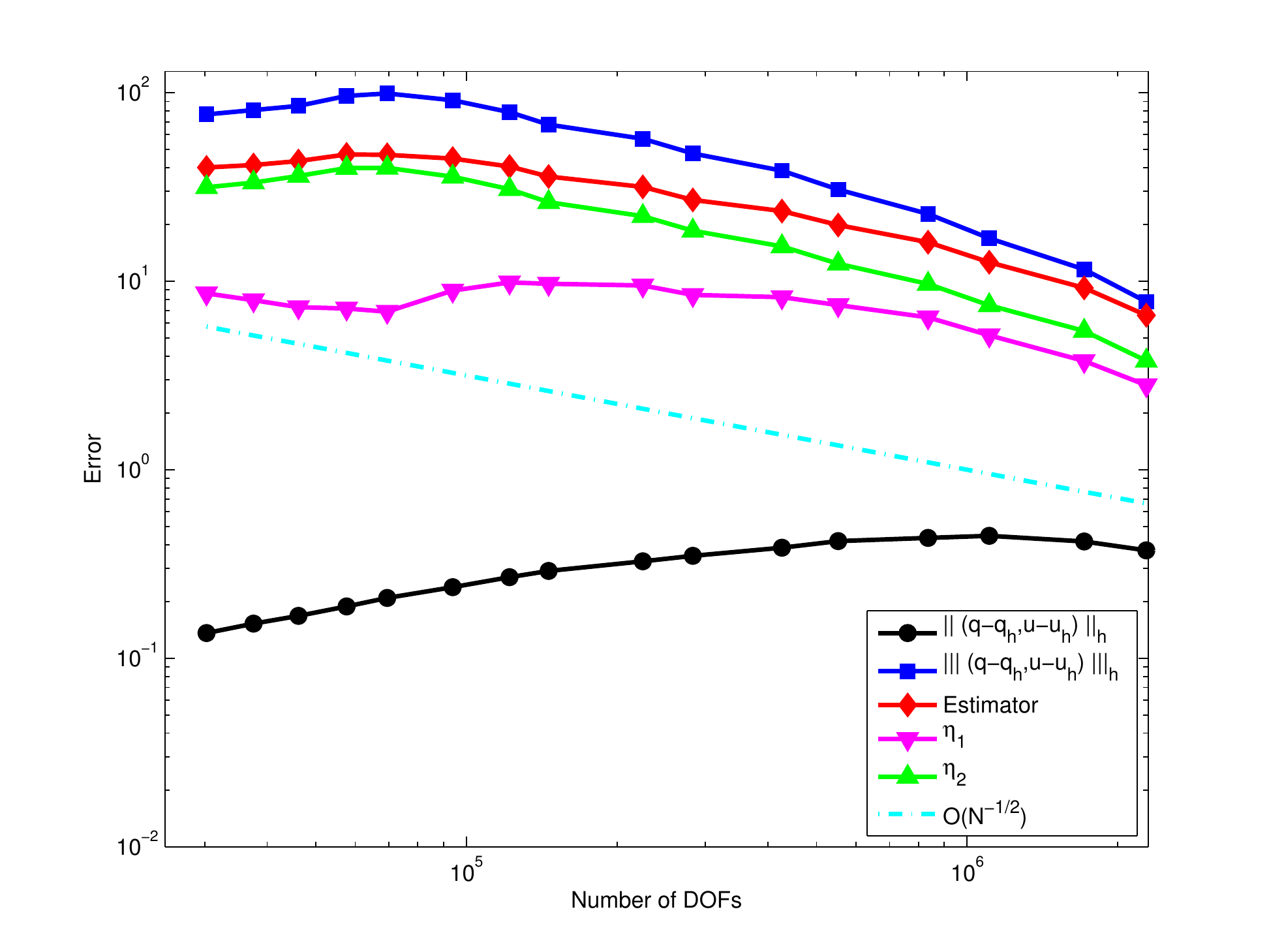}
\includegraphics[width=6cm,height=5.5cm]{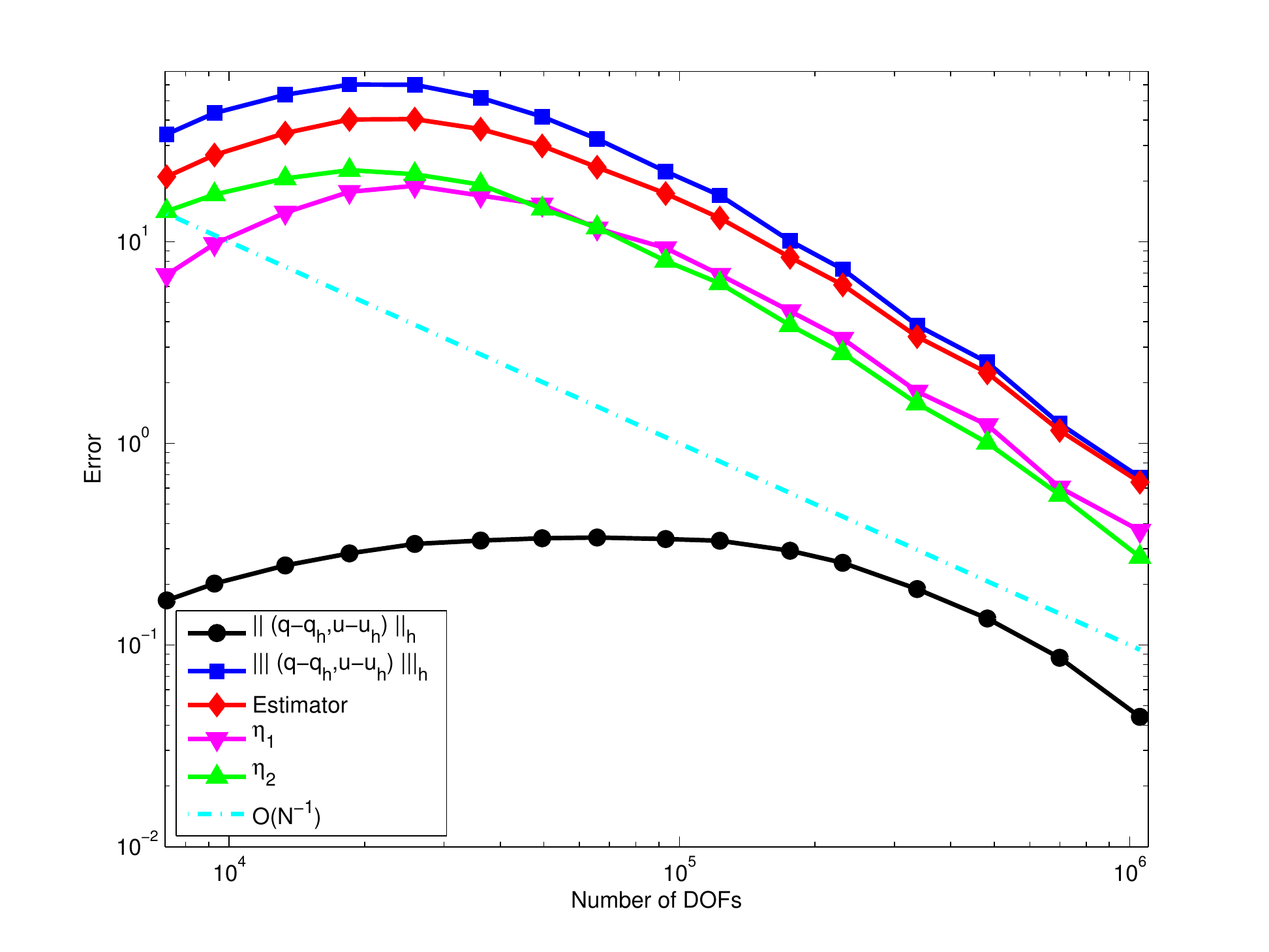}
\includegraphics[width=6cm,height=5.5cm]{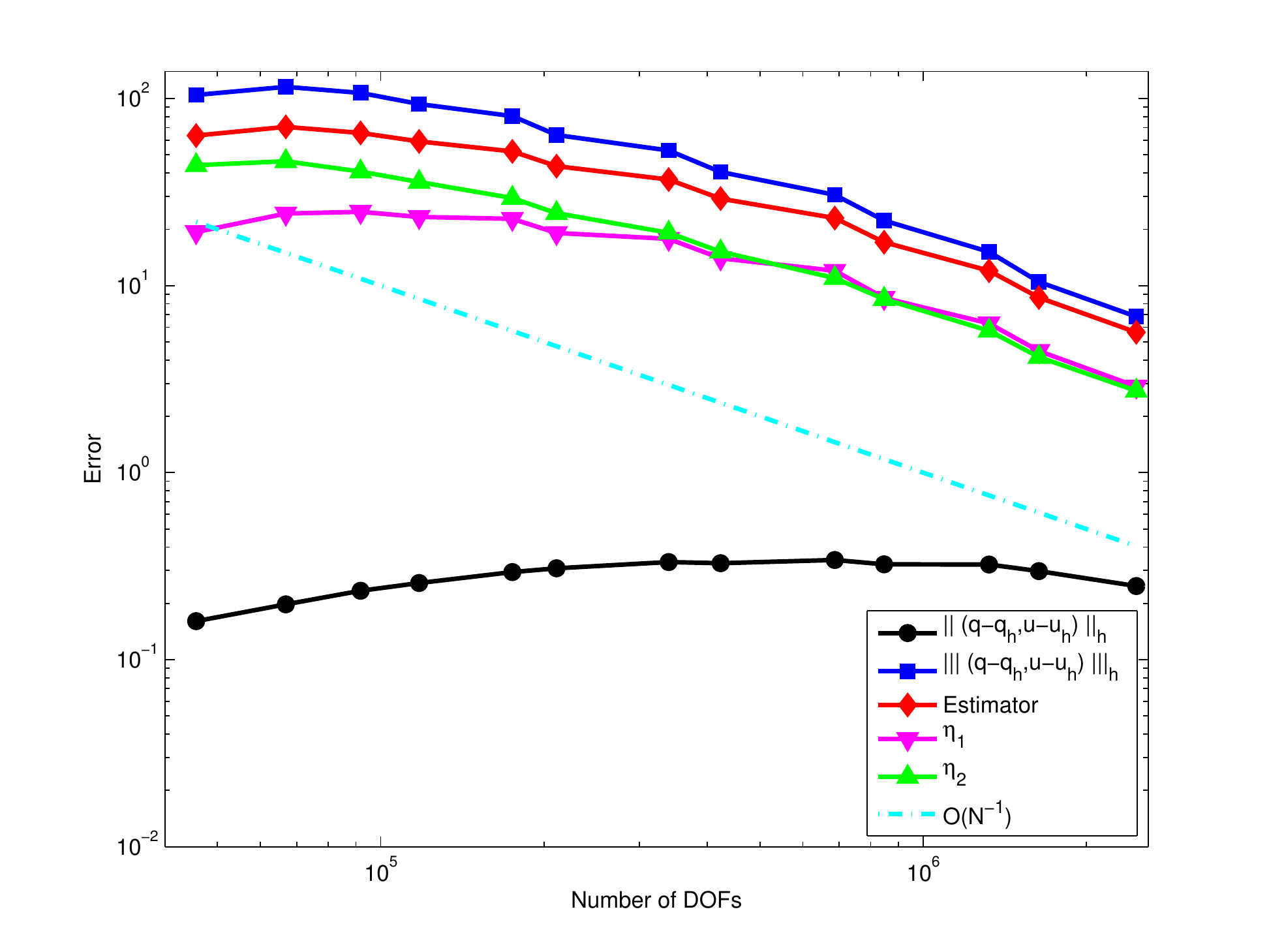}
\includegraphics[width=6cm,height=5.5cm]{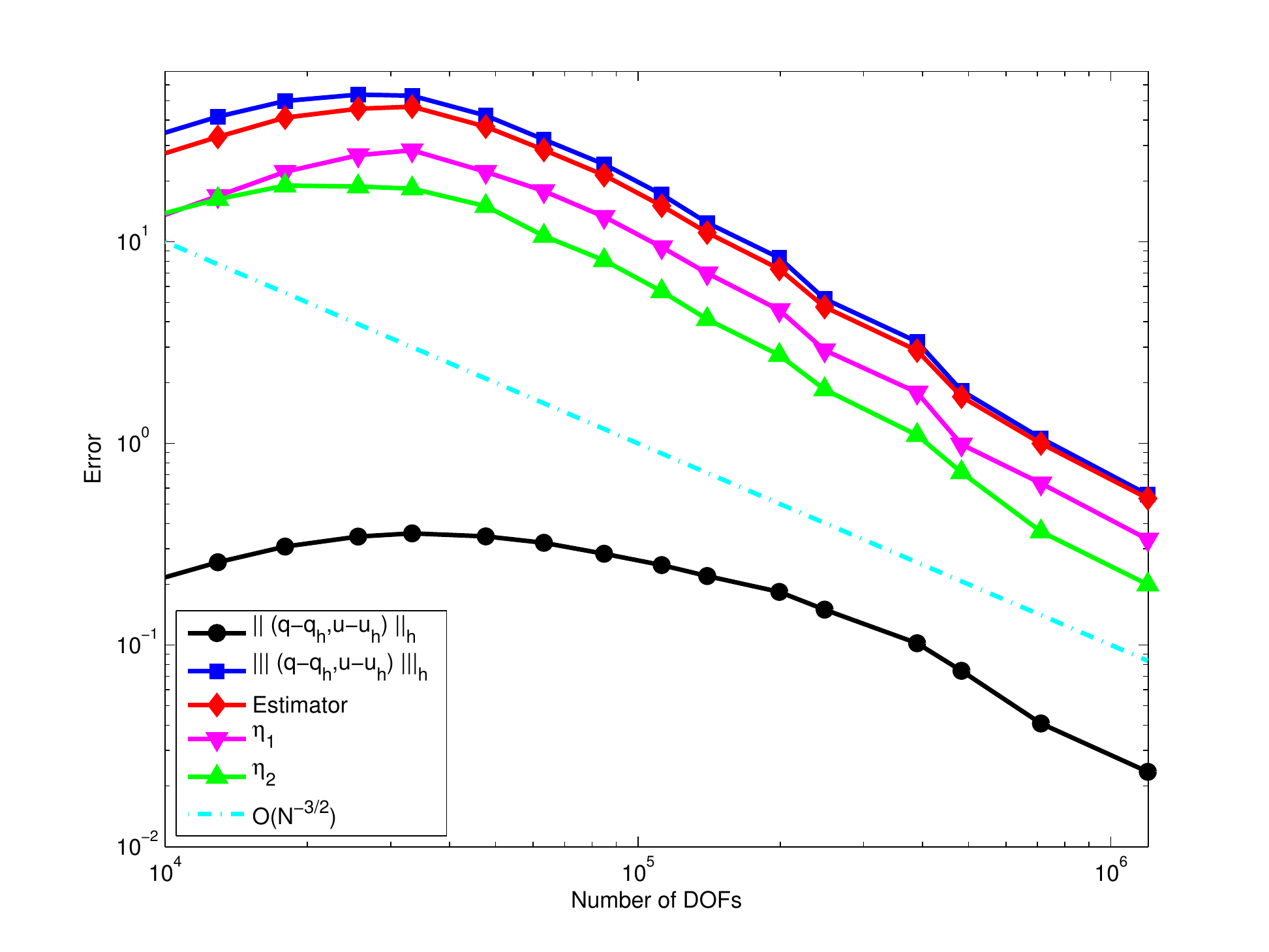}
\includegraphics[width=6cm,height=5.5cm]{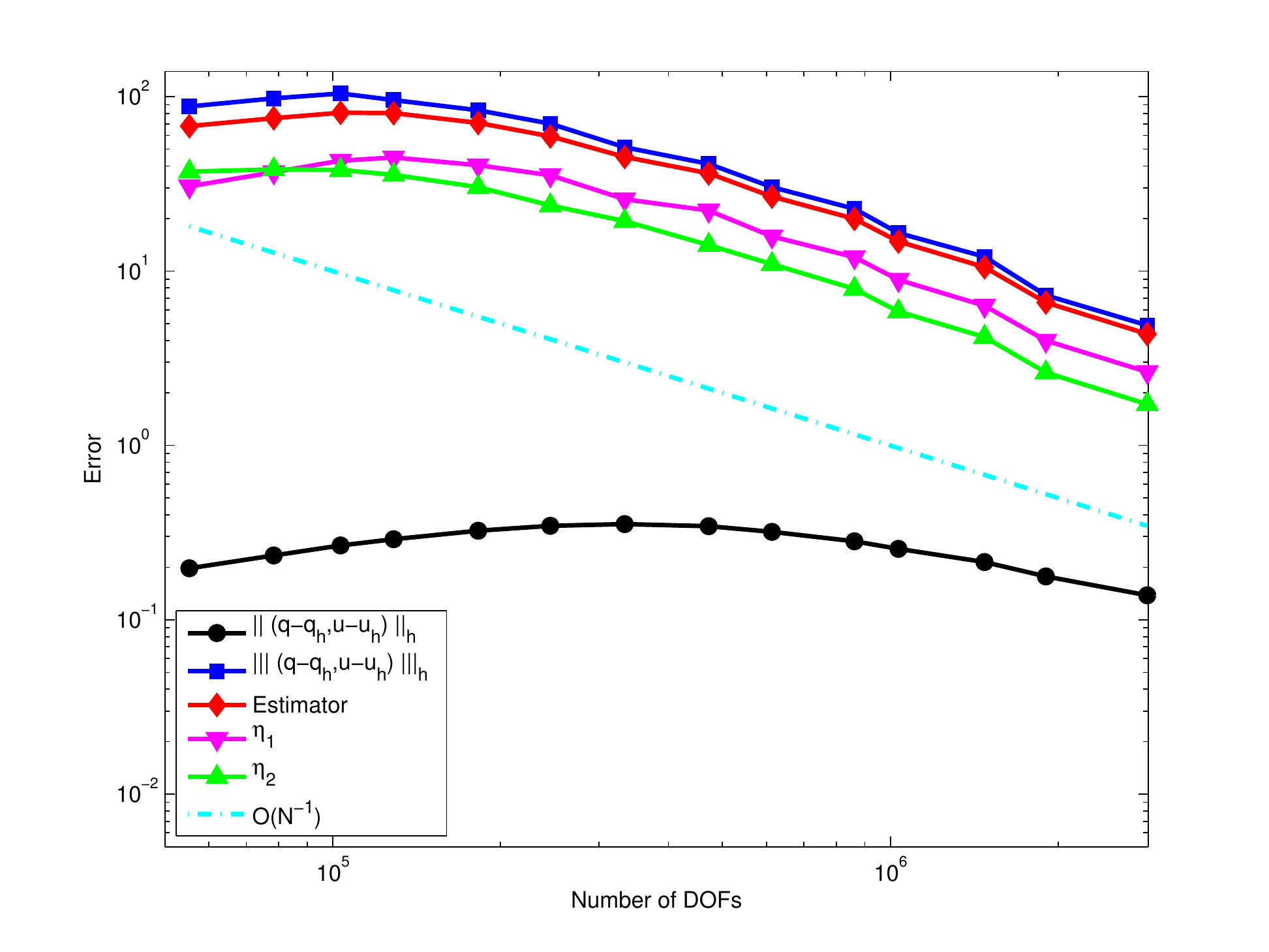}
\end{center}
\vspace{-0.2cm}
\caption{\footnotesize Convergence history of the adaptive HDG method. Left-Right: $\epsilon=10^{-5}, \epsilon=10^{-6}$. Top-Bottom: $P1$-$P3$.} \label{ex1-2}
\end{figure}

The solution develops boundary layers along the boundaries $x=1$ and $y=1$ for small $\epsilon$. The initial quasi-uniform mesh consists of 800 triangles and the initial mesh size $h_0=0.05$. Actually, our algorithm is robust for any coarser mesh. Figure \ref{ex1-1} displays the adaptively refined mesh and the corresponding approximate solution $u_h$ by 20 iterations of the adaptive HDG-P2 method for two cases $\epsilon = 10^{-4}$ and $\epsilon=10^{-5}$. One can observe that the mesh is always locally refined at the singularities along the boundaries $x=1$ and $y=1$, and the boundary layer solution can be captured on the adaptively refined mesh. In Figure \ref{ex1-2}, we show the error $\| (\Bq-\Bq_h,u-u_h) \|_h$, the total energy error $\interleave (\Bq-\Bq_h,u-u_h) \interleave_h$ (the total energy error 
is defined in (\ref{energy-error})), the a posteriori error estimators $\eta_1$ and $\eta_2$,
and the total a posteriori error estimator as functions of $N$ which is the number of degrees of freedom (DOFs) of $\widehat{u}_h$ for two cases $\epsilon = 10^{-5}$ and $\epsilon=10^{-6}$ by HDG-P1, HDG-P2 and HDG-P3 respectively. In the following, we always let DOFs refer to the DOFs of $\widehat{u}_h$. For the case $\epsilon = 10^{-5}$, the convergence results indicate the robustness of the proposed a posteriori error estimator and the almost optimal convergence rate $ O(N^{-p/2})$ for the adaptive HDG method when the number of DOFs is sufficiently large. Here, $p$ is the polynomial order. The convergence of $\eta_1$ and $\eta_2$ is similar as the total a posteriori error estimator.
For smaller $\epsilon = 10^{-6}$ in this example, although the convergence of  the error $\| (\Bq-\Bq_h,u-u_h) \|_h$ slows down, the convergence of the a posteriori error estimators and the total energy error is also almost $O(N^{-p/2})$ when $p=1,2$ and $O(N^{-1})$ when $p=3$ on the currently obtained meshes. 

}

\smallskip

\end{exm}


\begin{exm}
{\rm We consider an internal layer problem in \cite{Voh07}. We set $\Omega = [0,1] \times [0,1] , \beta = [0,1]^T, c=1$. 
The source term $f$ and the Dirichlet boundary condition are chosen such that 
\[
u(x,y) = 0.5\left(  1- {\rm tanh} \left( \frac{0.5-x}{\alpha}  \right)  \right)
\]
is the exact solution, where $\alpha$ is the width of the internal layer.

\smallskip

\begin{figure}[htbp]
\begin{center}
\includegraphics[width=6cm,height=5.5cm]{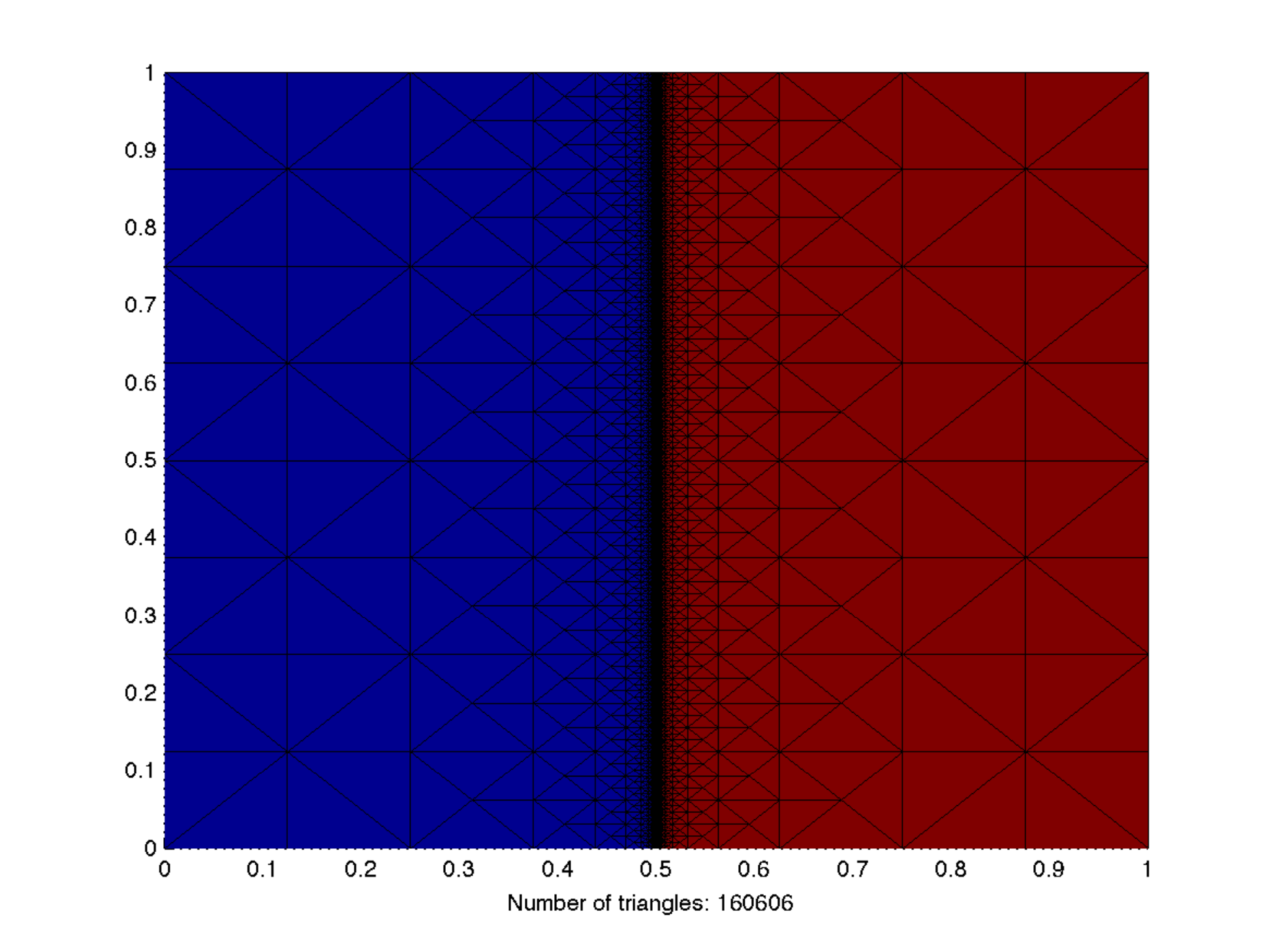}
\includegraphics[width=6cm,height=6cm]{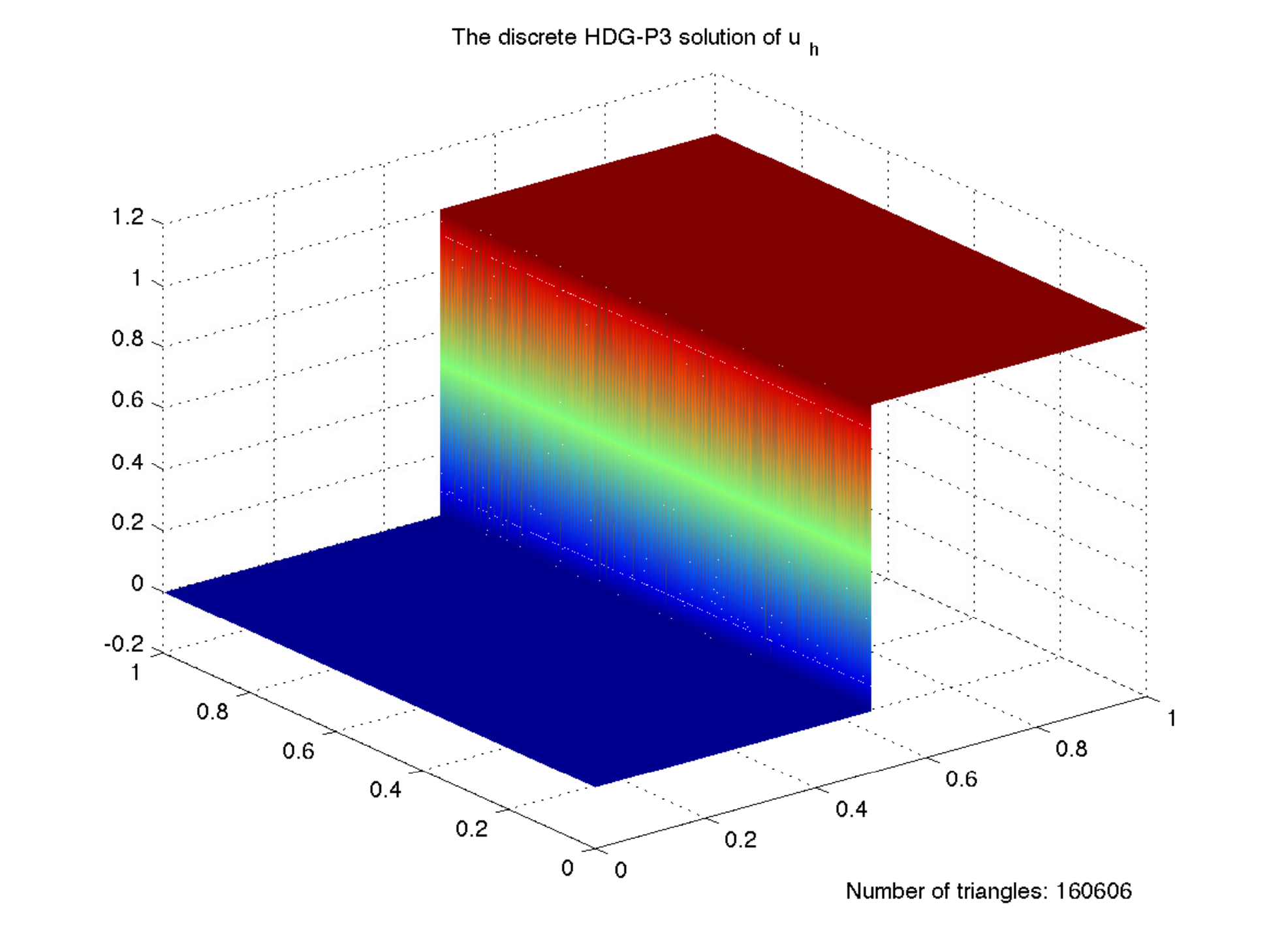}
\end{center}
\vspace{-0.2cm}
\caption{\footnotesize Adaptively refined mesh (Left) and 3D plot (Right) of the corresponding approximate solution $u_h$ by HDG-P3 for the case $\alpha= 10^{-4}$ and $\epsilon=10^{-6}$.} 
\label{ex2-1}
\end{figure}

\begin{figure}[htbp]
\begin{center}
\includegraphics[width=6cm,height=5.5cm]{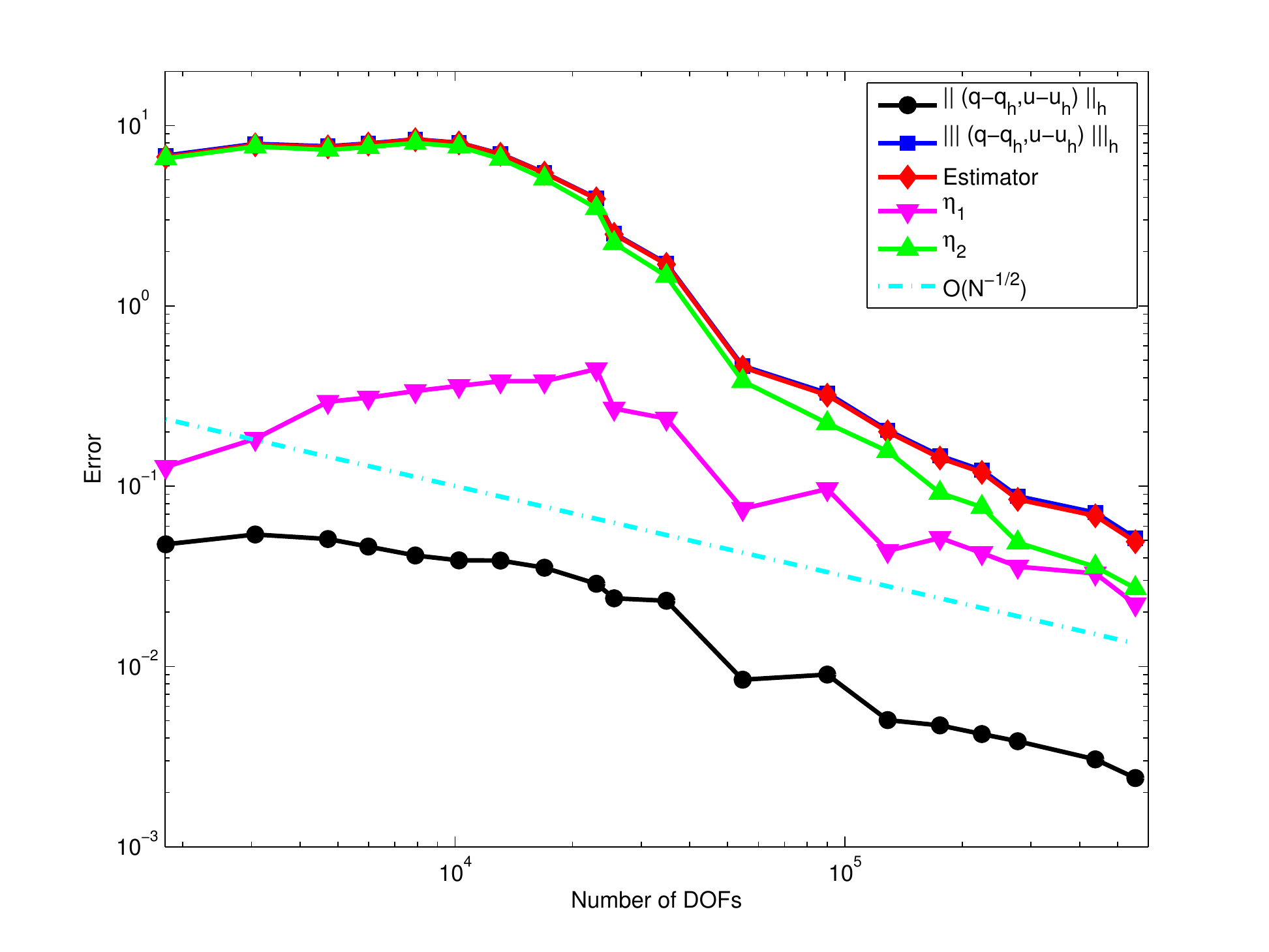}
\includegraphics[width=6cm,height=5.5cm]{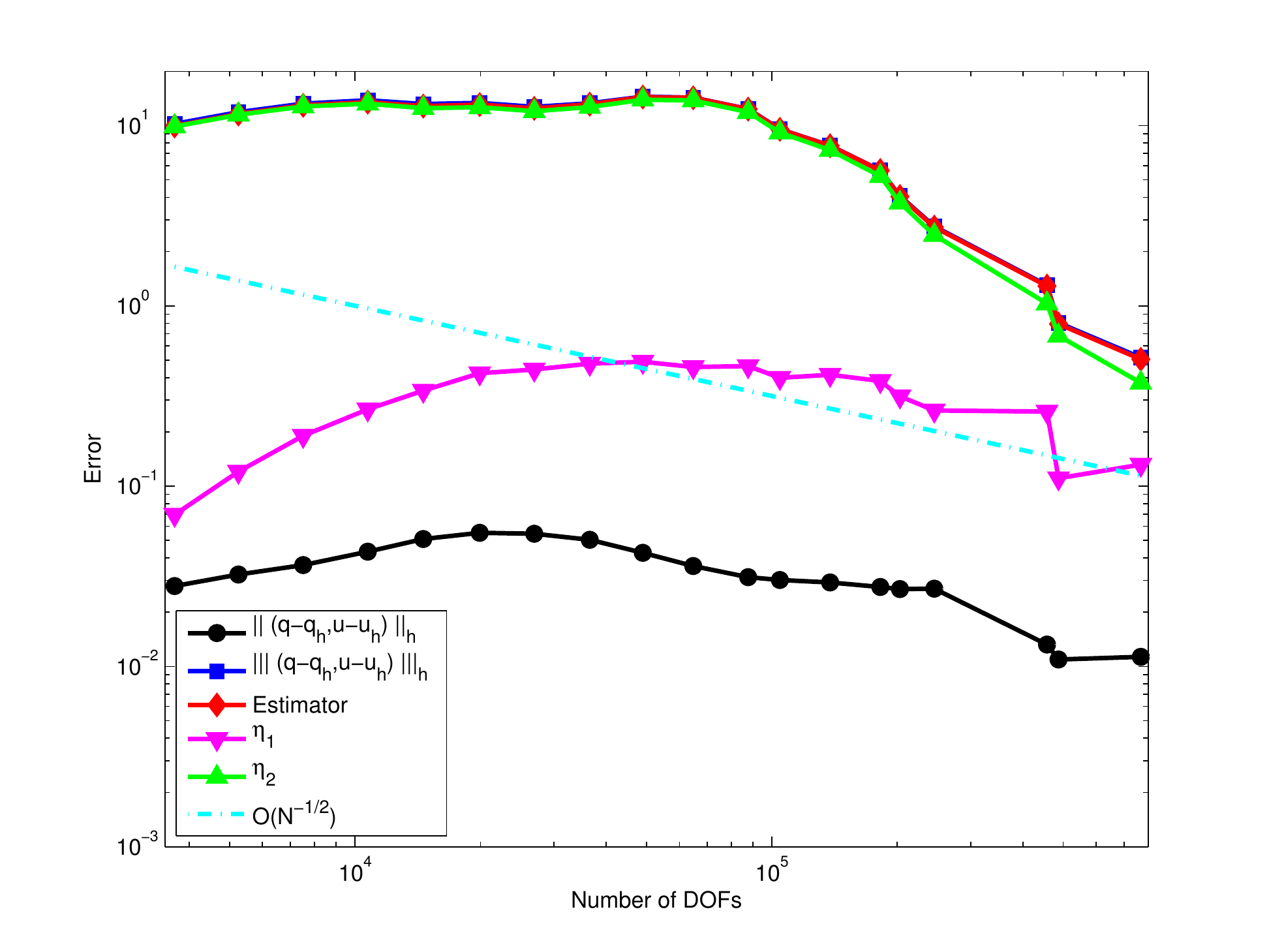}
\includegraphics[width=6cm,height=5.5cm]{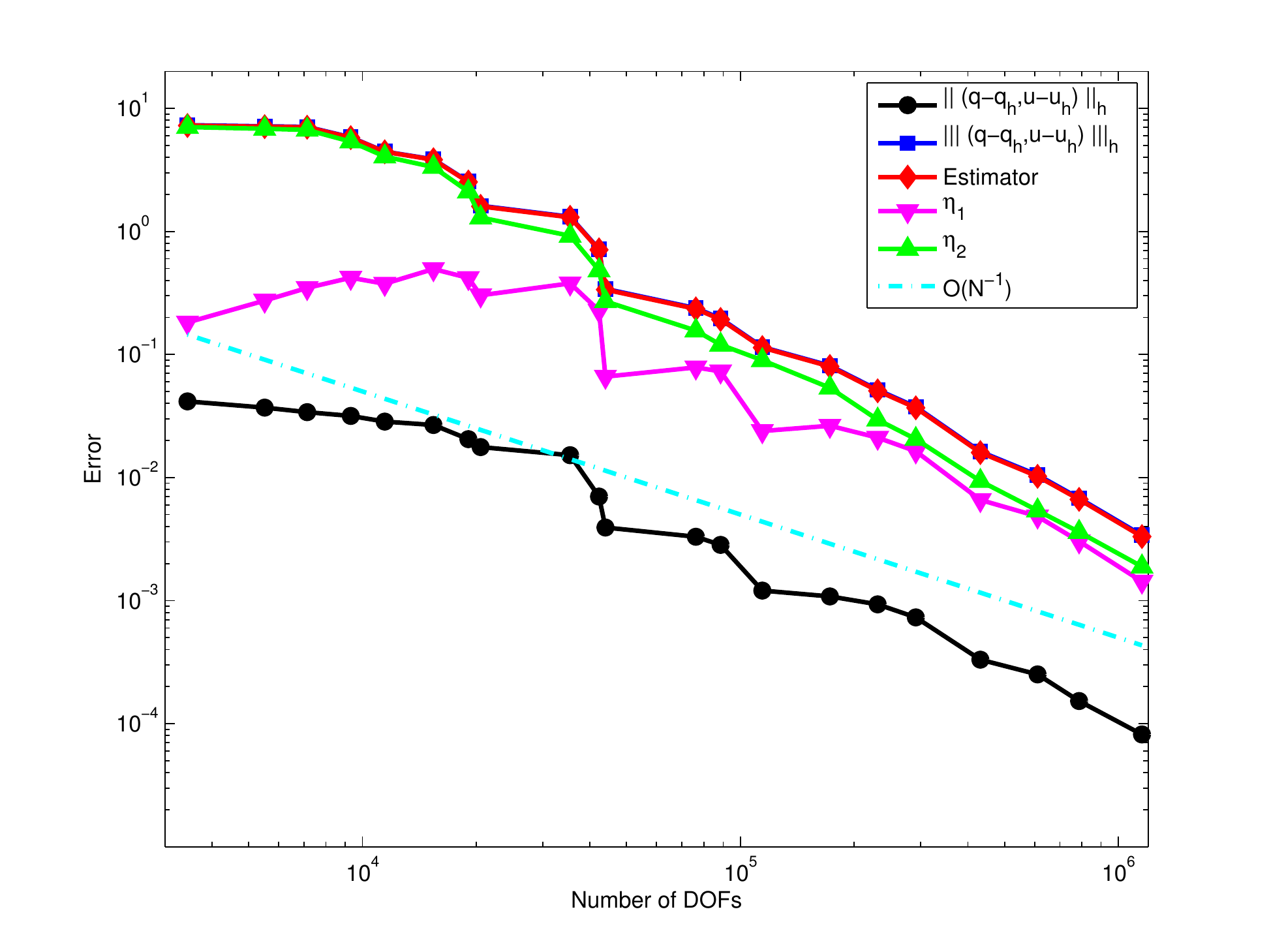}
\includegraphics[width=6cm,height=5.5cm]{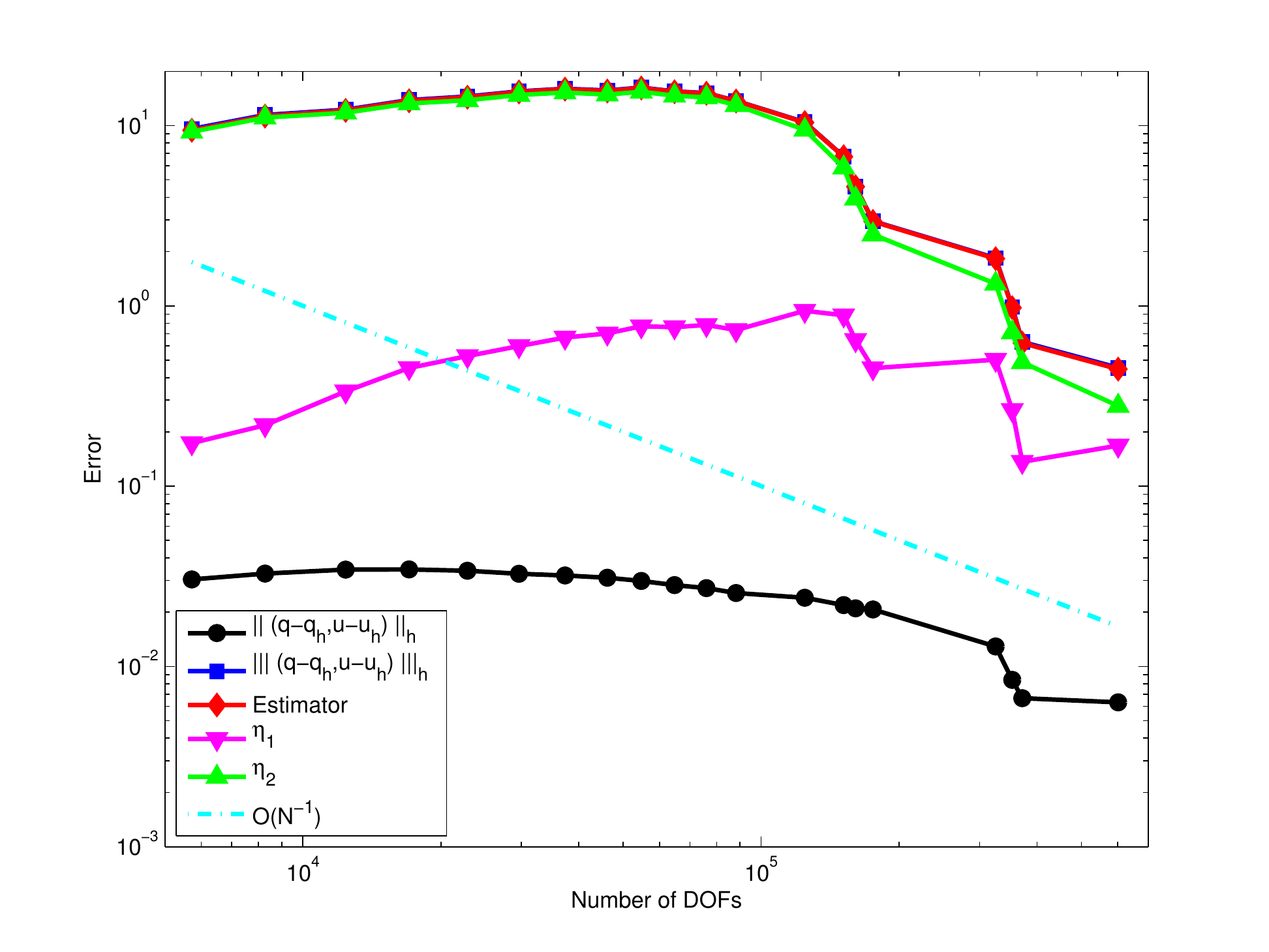}
\includegraphics[width=6cm,height=5.5cm]{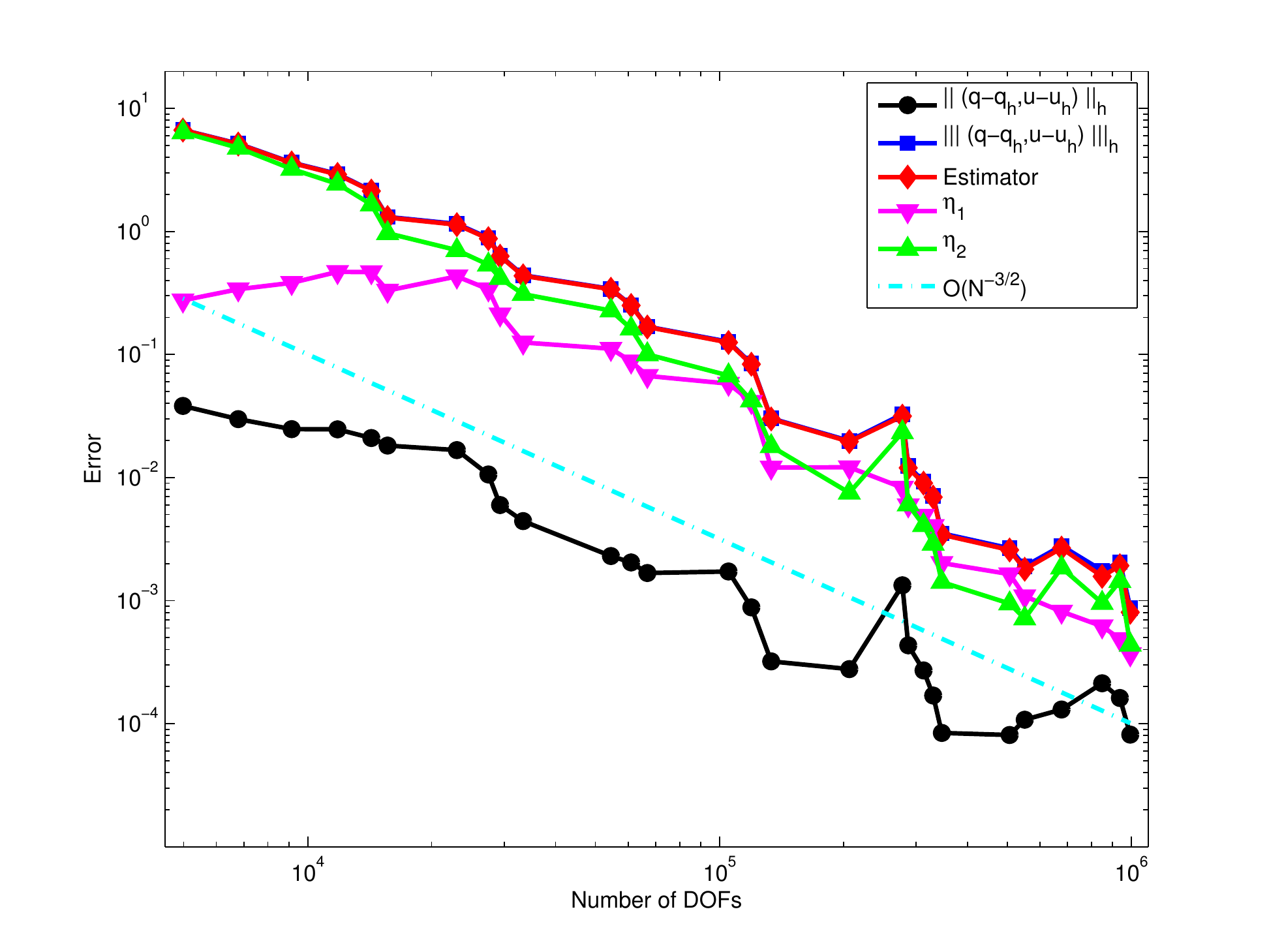}
\includegraphics[width=6cm,height=5.5cm]{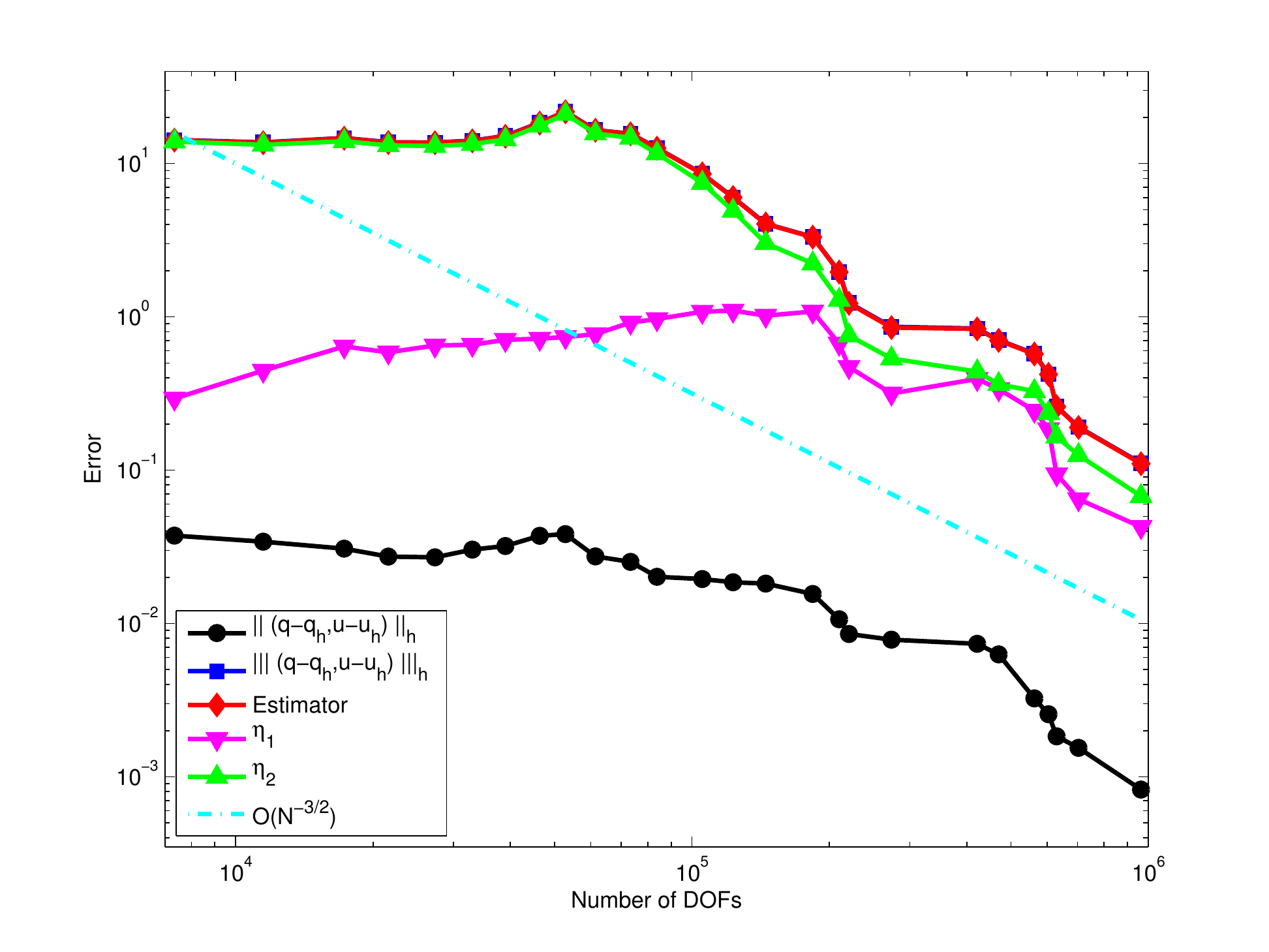}
\end{center}
\vspace{-0.2cm}
\caption{\footnotesize Convergence history of the adaptive HDG method. Left: $\alpha= 10^{-3}$ and $\epsilon =10^{-5}$. Right: $\alpha = 10^{-4}$ and $\epsilon=10^{-6}$. Top-Bottom: $P1$-$P3$.} \label{ex2-2}
\end{figure}

The solution of this problem possesses an internal layer along $x=0.5$. The initial quasi-uniform mesh 
consists of 128 triangles and the initial mesh size $h_0=0.12$. The graphs of Figure \ref{ex2-1} show the 
plots of the mesh and solution by 32 iterations of the adaptive HDG-P3 method for the case $\alpha= 10^{-4}$ 
and $\epsilon =10^{-6}$. We can see that the singularities of the solutions can also be captured near $x=0.5$ on the 
adaptively refined mesh. Figure \ref{ex2-2} shows the convergence of the corresponding errors. The robustness of 
the proposed a posteriori error estimator and the almost optimal convergence rate of the adaptive HDG method can be observed for HDG-P1, HDG-P2 and HDG-P3 when the number of DOFs is large enough. Moreover, we can also see that $\eta_1$ and $\eta_2$ converge similarly as the total a posteriori error estimator.

}

\end{exm}


\begin{figure}[htbp]
\begin{center}
\includegraphics[width=6cm,height=5.5cm]{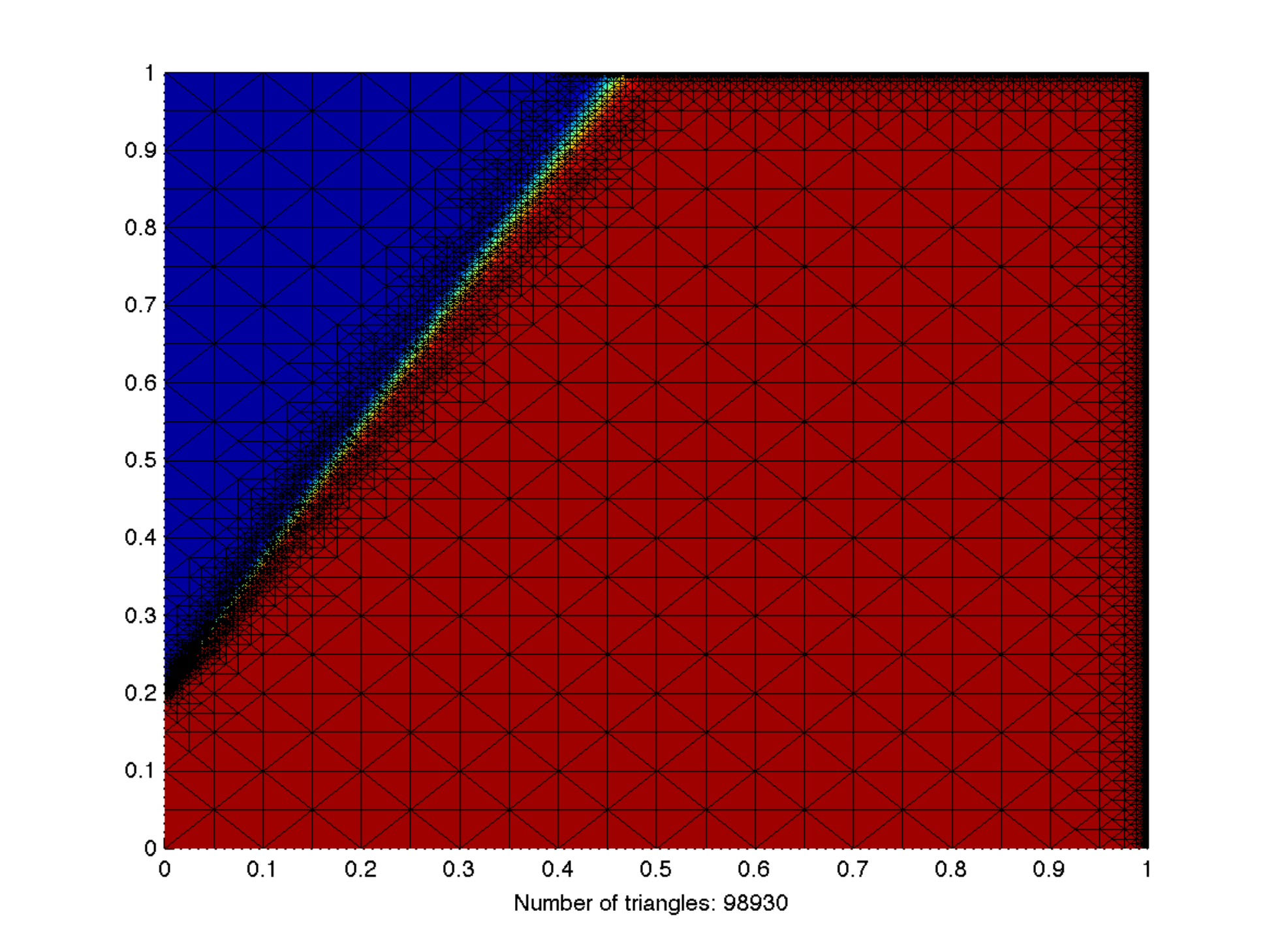}
\includegraphics[width=6cm,height=5.5cm]{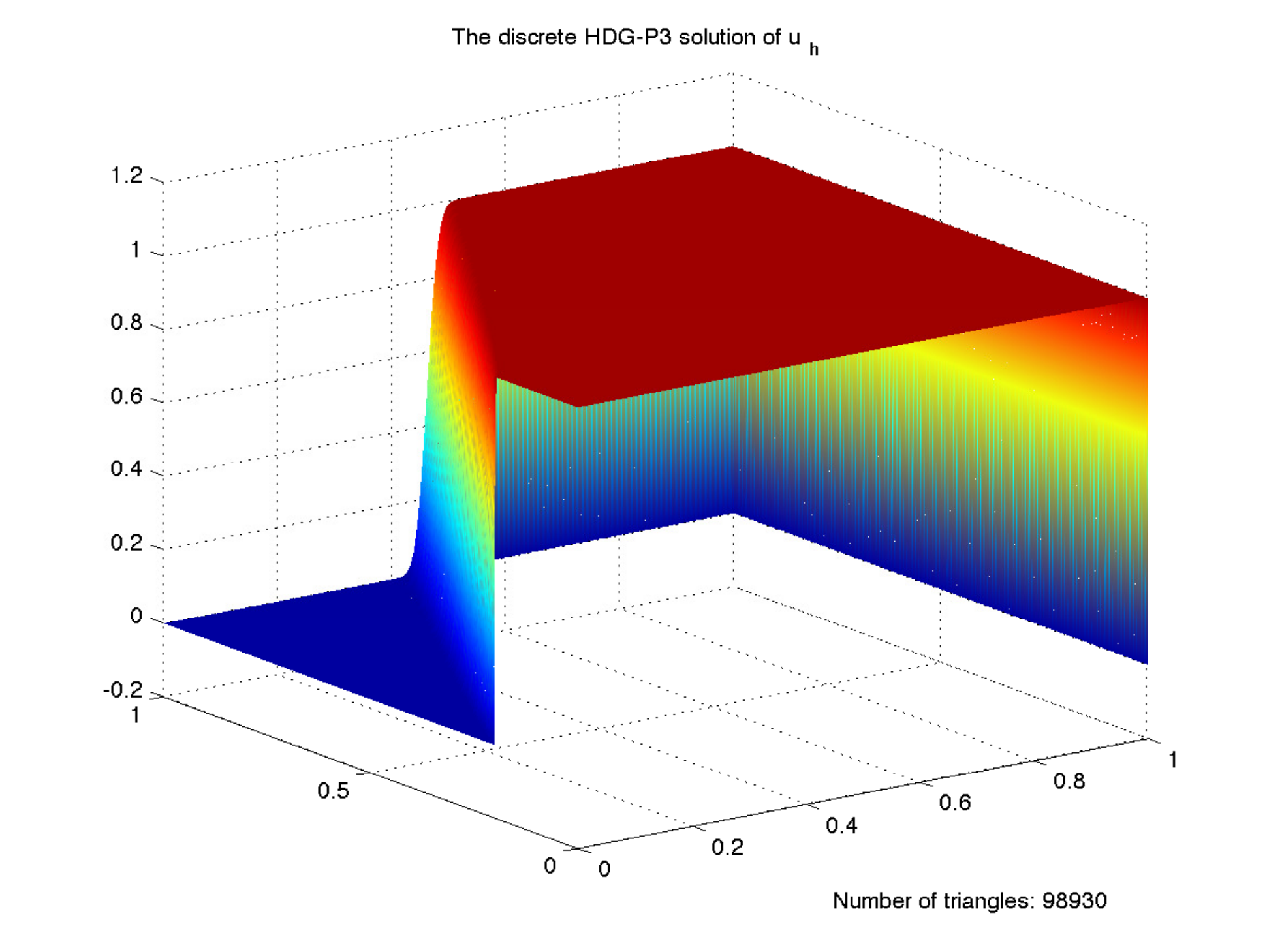}
\end{center}
\vspace{-0.2cm}
\caption{\footnotesize Adaptively refined mesh (Left) and the 3D plot (Right) of the corresponding approximate solution $u_h$ by HDG-P3 for the case $\epsilon =10^{-4}$.} \label{ex3-1}
\end{figure}
\begin{exm}
{\rm This example is also taken from \cite{AyusoMarini:cdf}. We set $\beta = [1/2,\sqrt{3}/2]^T,c=0$, the source term $f=0$ and the Dirichlet boundary conditions as follows:
\[
  u=
       \begin{cases}
          1 & \text{on $ \{y=0, 0\leq x \leq 1\} $ }, \\
          1& \text{on $ \{x=0, 0\leq y \leq 1/5\} $ },\\
       0 & \text{elsewhere
       }.
       \end{cases}
\]

\begin{figure}[htbp]
\begin{center}
\includegraphics[width=6cm,height=5.5cm]{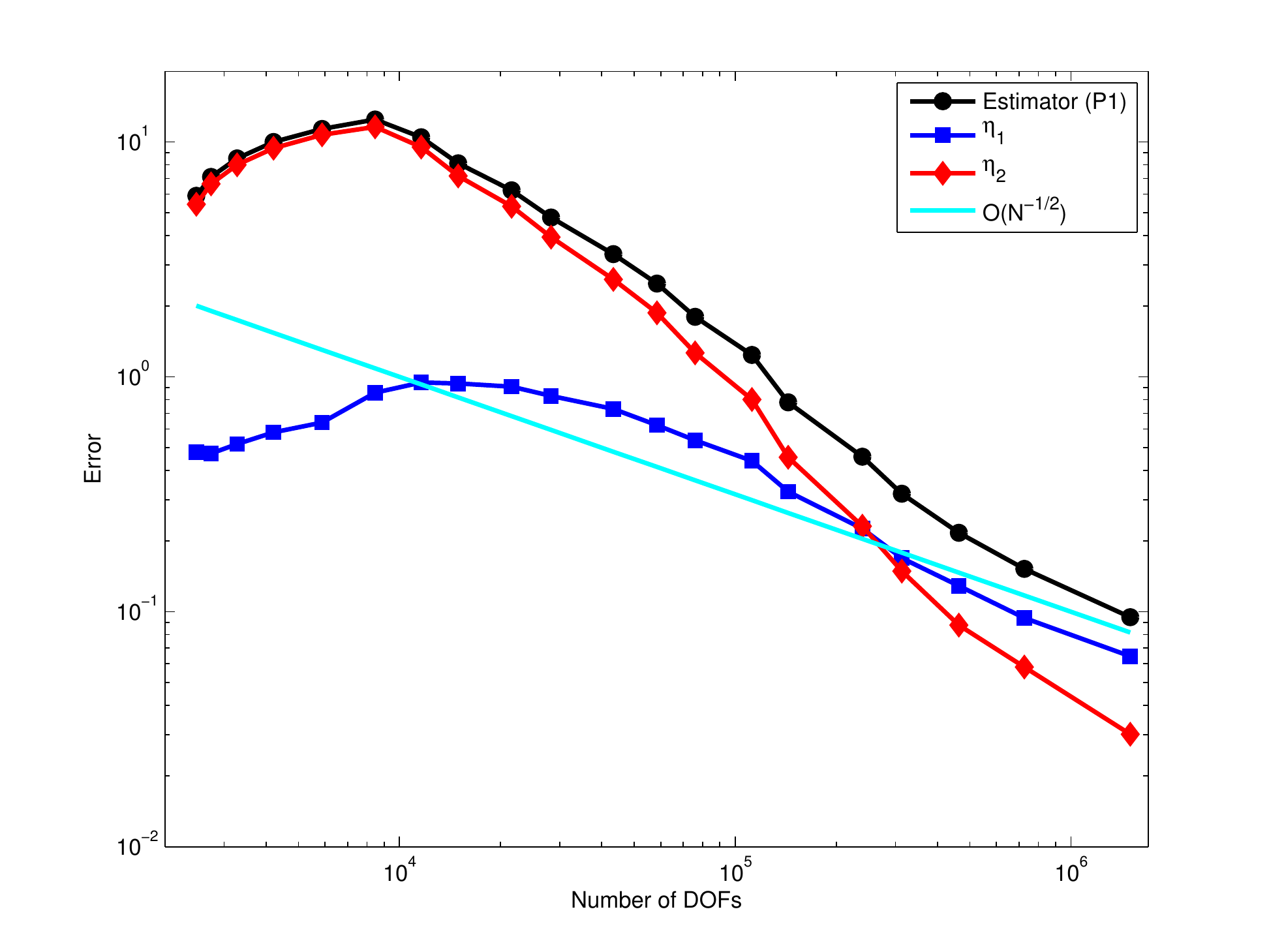}
\includegraphics[width=6cm,height=5.5cm]{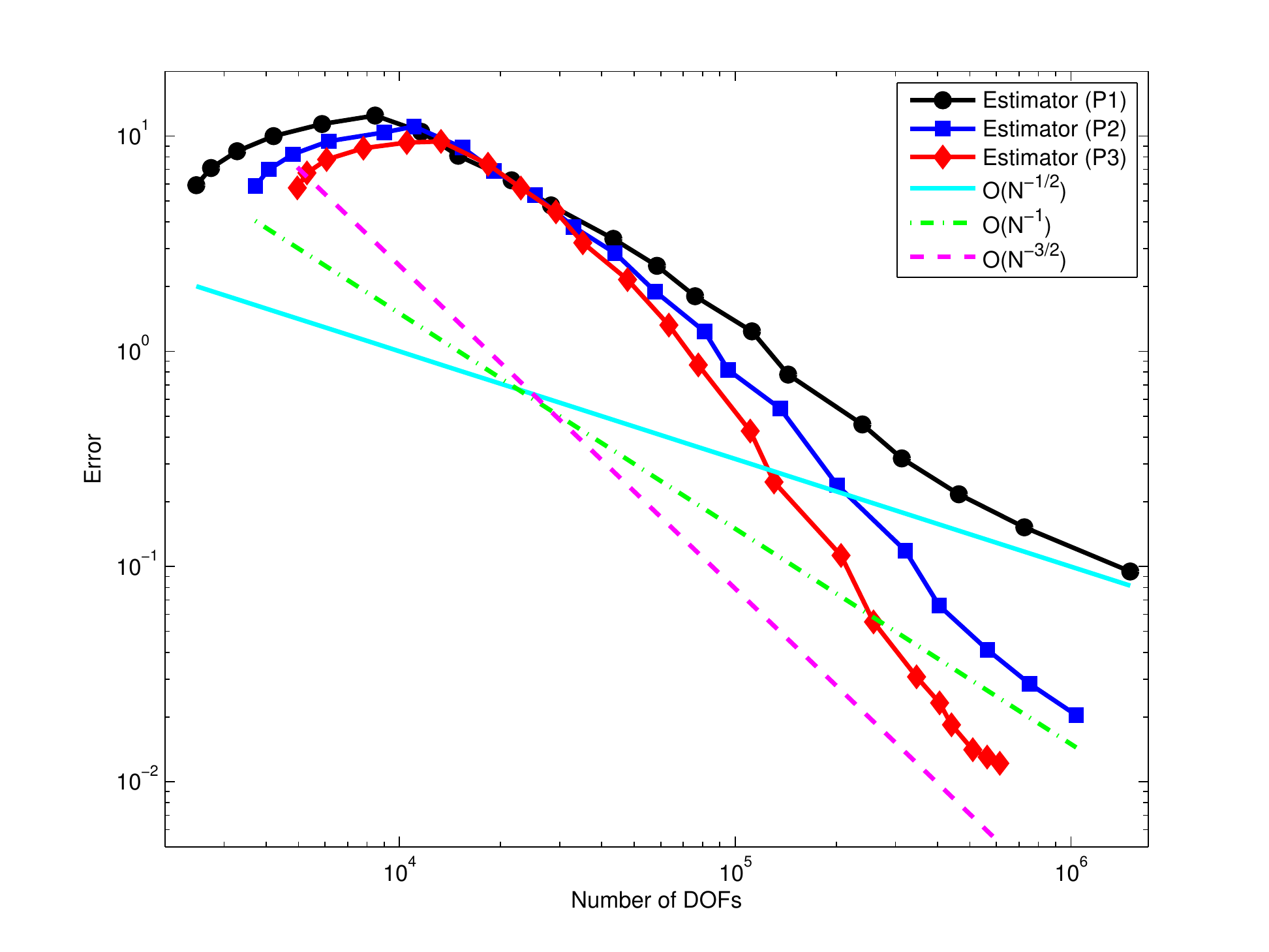}
\end{center}
\vspace{-0.2cm}
\caption{\footnotesize Convergence history of $\eta_1$, $\eta_2$ and the total a posteriori error estimator $\eta$ by HDG-P1 (Left) and comparing of convergence history of the estimator by HDG-P1, HDG-P2 and HDG-P3 (Right) for the case $\epsilon =10^{-4}$.} \label{ex3-2-1}
\end{figure}

\begin{figure}[htbp]
\begin{center}
\includegraphics[width=6cm,height=5.5cm]{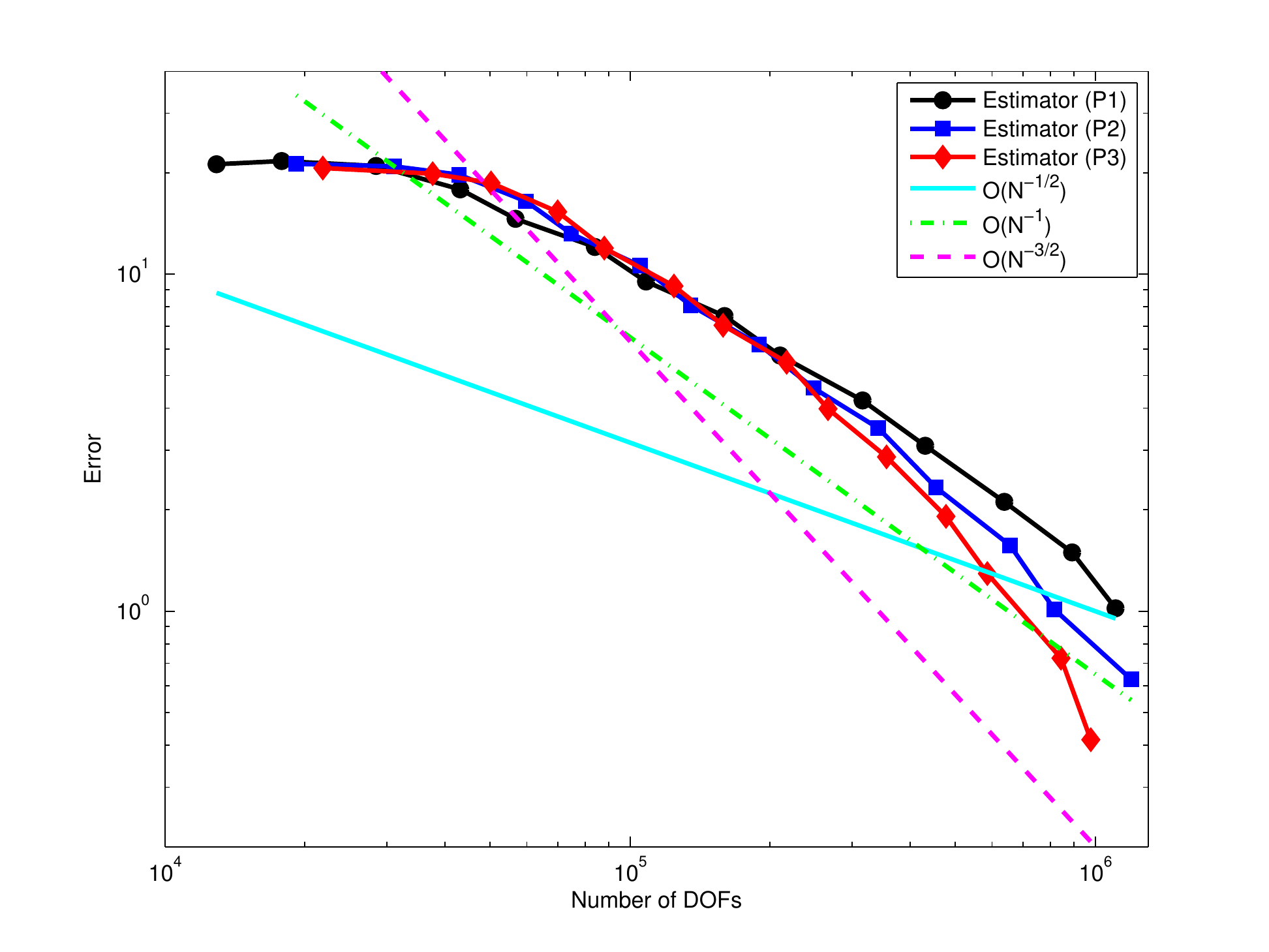}
\includegraphics[width=6cm,height=5.5cm]{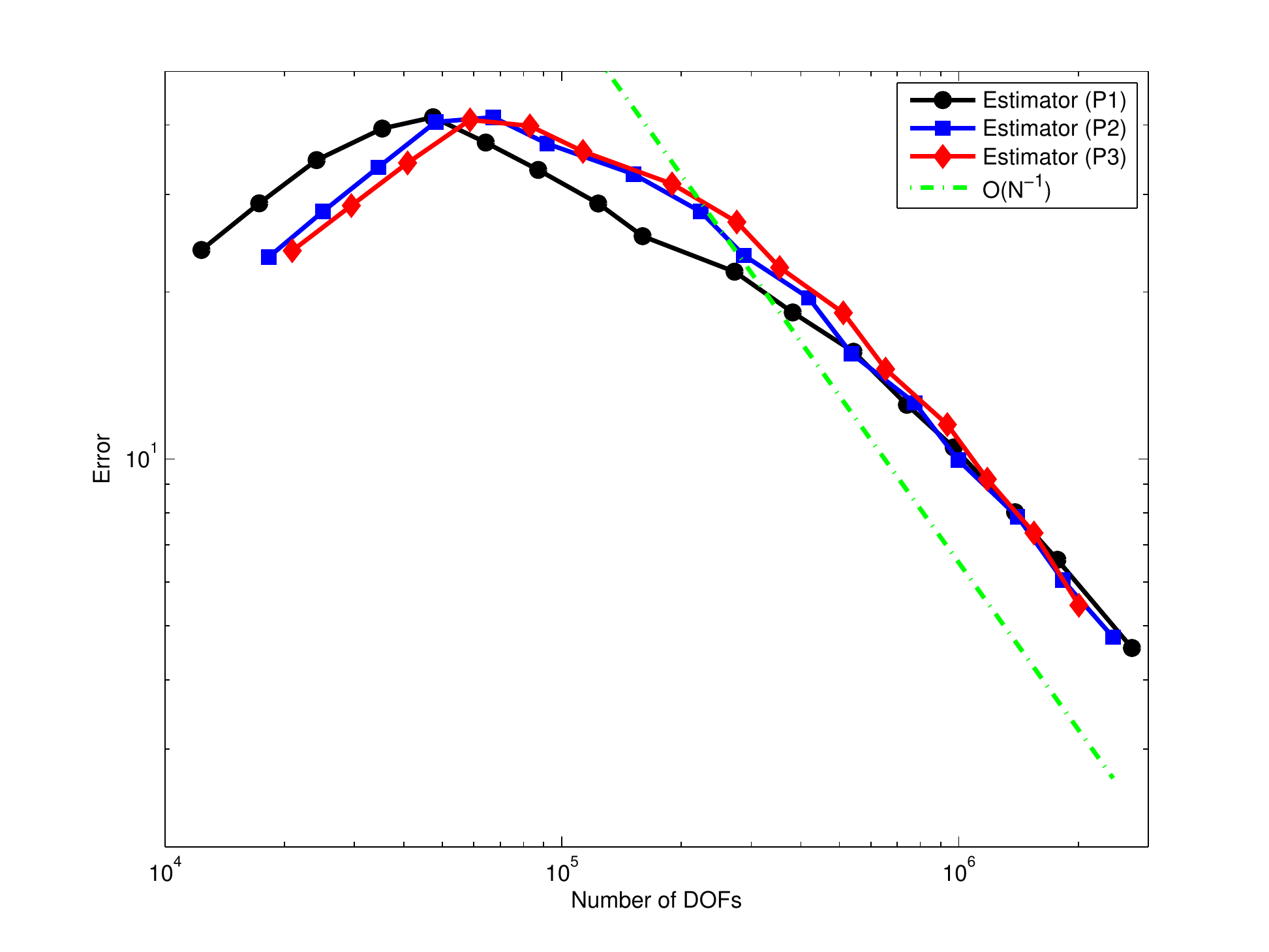}
\end{center}
\vspace{-0.2cm}
\caption{\footnotesize Convergence history of the estimator $\eta$ by HDG-P1, HDG-P2 and HDG-P3. Left: $\epsilon =10^{-5}$. Right: $\epsilon=10^{-6}$.} \label{ex3-2}
\end{figure}

The solution of this problem possesses both interior layer and outflow layer. The initial quasi-uniform mesh consists of 800 triangles and the initial mesh size $h_0=0.05$. The graphs of Figure \ref{ex3-1} show the adaptively refined mesh and the 3D plot of the corresponding approximate solution $u_h$ by 23 iterations of the adaptive HDG-P3 method for the case $\epsilon =10^{-4}$.
Figure \ref{ex3-1} indicates that both the interior and outflow layers can be captured by the adaptively refined mesh strategy. In particular, when the mesh size is $O(\epsilon)$ near the outflow layer and $O(\sqrt{\epsilon})$ near the interior layer, both interior and outflow layers can be captured well. Since there is no exact solution for this problem, we only show the convergence of the proposed a posteriori error estimator for the cases $\epsilon = 10^{-4},10^{-5}$ and $10^{-6}$  in Figure \ref{ex3-2-1} and Figure \ref{ex3-2}. For the cases $\epsilon = 10^{-4},10^{-5}$, the almost optimal convergence rate $O(N^{-p/2})$ of the estimator can be observed for HDG-P1, HDG-P2 and HDG-P3 when the number of DOFs is large enough, and the convergence rate is faster when $p$ is larger. From the left graph of Figure \ref{ex3-2-1}, we can also see that $\eta_1$ and $\eta_2$ converge similarly as the total a posteriori error estimator for the case $\epsilon=10^{-4}$ by HDG-P1. Actually, for other cases the convergence property of $\eta_1$ and $\eta_2$ is always similar. For smaller $\epsilon=10^{-6}$, the convergence of the total a posteriori error estimator is almost the same as $O(N^{-1})$ for $p=1,2,3$ on the currently obtained meshes.

}

\end{exm}

 \section*{Appendix A. Estimate of $I_{1}$ }
\renewcommand\theequation{A.1}
By the Cauchy-Schwarz and Young's inequalities, we get, for any $\delta>0$,
\begin{align}
\label{est_I1}
I_1 &\leq \frac{\delta}{2} \epsilon^{-1} \| \Bq - \Bq_h \|^2_{0,\Ct_h} + \frac{1}{2\delta} \| \nabla \psi e^{-\psi} \|^2_{L^\infty(\Omega)} 
\epsilon \| u_h -u^{\Ci}_h \|^2_{0,\Ct_h}\\
\nn
&\quad +  \frac{\delta}{2} \epsilon^{-1} \| \Bq - \Bq_h \|^2_{0,\Ct_h} +\frac{1}{2\delta} \|\varphi\|^2_{L^\infty(\Omega)} \epsilon^{-1} \| \Bq_h
+\epsilon \nabla u^{\Ci}_h \|^2_{0,\Ct_h}\\
\nn
&\leq \delta \epsilon^{-1} \| \Bq - \Bq_h \|^2_{0,\Ct_h} +  \frac{1}{2\delta} \| \nabla \psi e^{-\psi} \|^2_{L^\infty(\Omega)} \epsilon \| u_h -u^{\Ci}_h \|^2_{0,\Ct_h}\\
\nn
&\quad + \frac{1}{\delta}\|\varphi\|^2_{L^\infty(\Omega)} \left(  \epsilon^{-1} \| \Bq_h+\epsilon \nabla u_h \|^2_{0,\Ct_h} + \epsilon \|\nabla u_h -  \nabla u^{\Ci}_h  \|^2_{0,\Ct_h}  \right).
\end{align}
     
\section*{Appendix B. estimate of $I_{4}$}
\renewcommand\theequation{B.1}
Now we consider the estimate of $I_4$. It is clear that 
\[
-\left( \bbeta \cdot \nabla(u-u^{\Ci}_h) , \varphi( u-u^{\Ci}_h )  \right)_{\Ct_h} = \frac{1}{2}\left(  \bbeta \cdot \nabla \varphi, (u-u^{\Ci}_h)^2\right)_{\Ct_h} + \frac{1}{2}\left(  {\rm div}\bbeta \, , \varphi(u-u^{\Ci}_h)^2\right)_{\Ct_h},
\]
which can be obtained by integration by parts and $\langle \bbeta \cdot \Bn , \varphi( u-u^{\Ci}_h )^2\rangle_{\partial \Ct_h}=0$. Then, we have
\begin{align}
I_4&=\left( \bbeta \cdot \nabla \varphi(u-u_h),u_h - u^{\Ci}_h  \right)_{\Ct_h} + \frac{1}{2}\left( \bbeta \cdot \nabla \varphi(u^{\Ci}_h-u_h ) ,u^{\Ci}_h-u_h\right)_{\Ct_h}\nn\\
&\quad + 2\big( (c-\frac{1}{2}{\rm div}\bbeta)(u^{\Ci}_h-u_h  ), \varphi(u-u_h) \big)_{\Ct_h} - 
\big( (c-\frac{1}{2}{\rm div}\bbeta)(u^{\Ci}_h-u_h  ), \varphi(u^{\Ci}_h-u_h) \big)_{\Ct_h}. \nn
\end{align}
Appling the Cauchy-Schwarz and Young's inequalities, we get
\begin{align}
\label{est_I4}
I_4 \leq \frac{1}{2}(C^d_5+2C^d_6)(\frac{1}{\delta}+1) \| u_h - u^{\Ci}_h \|^2_{0,\Ct_h} + (\frac{1}{2}C^d_5+C^d_6)\delta \| u- u_h \|^2_{0,\Ct_h},
\end{align}
where $C^d_5 = \| \bbeta \cdot \nabla \varphi \|_{L^\infty(\Omega)}, C^d_6=\|\varphi(c-\frac{1}{2}{\rm div}\bbeta)\|_{L^\infty(\Omega)}$.

\section*{Appendix C. Proof of Lemma \ref{lemma_control_eta2}}
\begin{proof}
We begin by the estimate $\alpha^2_T \| R_h \|^2_{0,T} \leq 2 \alpha^2_T \| P_WR_h \|^2_{0,T}  + 2 osc^2_h(R_h,T)$.
Note that for any $w\in H^1_0(T)$, we have
\[
(R_h,w)_T = ( {\rm div}(\Bq-\Bq_h)+\bbeta \cdot \nabla (u-u_h),w  )_T + (c(u-u_h),w)_T.
\]
Moreover, for the element bubble function $B_T$, we have
\[
\| P_W R_h \|^2_{0,T} \thickapprox \int_T B_T (P_W R_h)^2 \thickapprox \| B_T P_W R_h \|^2_{0,T} .
\]
Here, we use $A \thickapprox B$ if $C B \leq A \leq CB$ with positive constant $C$.
Taking $w = B_T P_W R_h$, we derive that
\begin{align*}
\| P_W R_h \|^2_{0,T} &\thickapprox(P_W R_h ,w)_T = ( R_h ,w)_T - (R_h - P_W R_h ,w)_T\\
& = ( {\rm div}(\Bq-\Bq_h)+\bbeta \cdot \nabla (u-u_h),w  )_T + (c(u-u_h),w)_T- (R_h - P_W R_h ,w)_T.
\end{align*}
Then, we can conclude the proof is complete by the above estimates and the Young's inequality.
\end{proof}


\end{document}